\documentclass[12pt,a4paper]{article}
\usepackage{amsfonts}

\usepackage{mathrsfs}
\usepackage[pagewise]{lineno}

\usepackage{stmaryrd}
\usepackage{amsmath}
\usepackage{amssymb}
\usepackage{xcolor}
\usepackage{bbm}
\usepackage{mathrsfs}
\usepackage{graphicx}
\usepackage{lipsum}

\usepackage{enumerate}
\usepackage{theorem}
\usepackage{indentfirst}
\makeatletter
\def\tank#1{\protected@xdef\@thanks{\@thanks
        \protect\footnotetext[0]{#1}}}
\def\bigfoot{

    \@footnotetext}
\makeatother

\newcommand{\ea}{\end{array}}

\newtheorem{theorem}{Theorem}[section]
\newtheorem{hypothesis}{Hypothesis}[section]

\newtheorem{lemma}{Lemma}[section]
\newtheorem{definition}{Definition}[section]
\newtheorem{Rem}{Remark}[section]

{\theorembodyfont{\rmfamily}
}

\newenvironment{proof}{Proof.}

\def \eref#1{\hbox{(\ref{#1})}}

\allowdisplaybreaks

\begin{document}
\title{{\Large \bf McKean-Vlasov SPDEs driven by Poisson random measure: Well-posedness and large deviation principle }}
\author{{Yuhang Jiang$^{a}$},~~{Jinming Li$^{b}$},~~{Shihu Li$^{a}$}
\\
 \small  $a.$ School of Mathematics and Statistics, Jiangsu Normal University, 221116, Xuzhou, China \\
\small $b.$ School of mathematical sciences, Shanghai Jiao Tong University, 200240, Shanghai, China}

\date{}
\maketitle
\begin{abstract}
In this work, we investigate the McKean-Vlasov stochastic partial differential equations driven by Poisson random measure.
By adapting the variational framework, we prove the well-posedness and large deviation principle for a class of McKean-Vlasov stochastic partial differential equations with monotone coefficients. The main results can be applied to quasi-linear McKean-Vlasov equations such as distribution dependent stochastic porous media equation and stochastic $p$-Laplace equation. Our proof is based on the weak convergence approach introduced by Budhiraja et al.~for Poisson random measures, the time discretization procedure and relative entropy estimates. In particular, we succeed in dropping the compactness assumption of embedding in the Gelfand triple in order to deal with the case of bounded and unbounded domains in applications.
\end{abstract}
\vspace{2mm}


\section{Introduction}

McKean-Vlasov stochastic partial differential equations (SPDEs), also called mean-field SPDEs or distribution dependent SPDEs,
have attracted more and more attention in recent years. 
Roughly speaking, these are SPDEs where their coefficients also depend on the distribution of solutions. McKean-Vlasov SPDEs can be used to characterize the limiting behaviors of $N$-interacting particle systems of mean-field type while $N$ goes to infinity. This kind of macroscopic limit behavior for the interacting particle system is usually referred to the propagation of chaos in the literature.
For example, Kallianpur and Xiong \cite{KX} studied the propagation of chaos problem for the interacting system of infinite-dimensional SDEs driven by Poisson random measures. Shen et al.~\cite{SSZZ} studied the large $N$ limits of  $O(N)$  linear sigma model posed over $\mathbb{T}^d$ for $d=1,2$, where the limit is characterized by a McKean-Vlasov singular SPDE.
Criens in \cite{C1} established the propagation of chaos for the weakly coupled SPDEs whose coefficients satisfy linear growth and Lipschitz type conditions.  Angeli et al. \cite{ABKO} studied the
well-posedness and stationary solutions of McKean-Vlasov SPDEs.
We also refer to \cite{CKS,ES1,HHL,HLL23,HKXZ,L18,MSSZ} for further studies concerning McKean-Vlasov SPDEs.

The large deviation principle (LDP) mainly investigates the
asymptotic property of remote tails of a family of probability distributions.
Due to its wide applications in extremal events arising
in risk management, mathematical finance, statistical mechanics, quantum physics
and many other areas, large deviation theory has become an important topics in probability theory. The small
perturbation type LDP (also called Freidlin-Wentzell's LDP) for stochastic differential equations in finite dimensional case is established by Freidlin and Wentzell in the pioneering work \cite{FW}, which has been widely investigated in recent years, we refer the readers to the classical monographs \cite{DZ,S,V66} and
references therein for the detailed exposition on the background and applications of large deviation theory.

 The Freidlin-Wentzell type LDP for McKean-Vlasov SDEs and SPDEs driven by Gaussian noise has been extensively investigated in the literature. Herrmann et al.~\cite{HIP} obtained the LDP  in path space with the uniform topology under the superlinear growth and coercivity
hypothesises of the drift. Dos Reis et al.~\cite{DST} also investigated the LDP in both uniform and H\"{o}lder topologies via assuming that the coefficients satisfy some extra time H\"{o}lder continuity conditions. Delarue et al. \cite{DLR20} first established
a weak LDP for McKean-Vlasov systems in the presence of common noise, and on the absence of common noise, they successfully upgraded this to a full LDP and obtain new concentration estimates for McKean-Vlasov systems.
Recently, based on the modified weak convergence criteria in \cite{LSZZ}, Hong et al. \cite{HLL21} studied the LDP for a class of nonlinear McKean-Vlasov SPDEs driven by multiplicative Gaussian noise, where the variational framework is similar to the framework of this article, and this result is extended to more general local monotonic frameworks in \cite{HHL}. The interested readers can refer to \cite{FYY23,HLLS,SZW} and references therein for recent progress on LDP for McKean-Vlasov equations.

Note that all of the above work is about large deviations of McKean-Vlasov equations with Gaussian noise.
However, Gaussian noise inadequately characterizes certain real-world models. From a particle systems perspective, both individual particles and their collective population often exhibit sudden jumps in numerous scenarios.
Thus, it is natural to consider the McKean-Vlasov SPDE with L\'{e}vy noise through utilizing Poisson random measures, and ask what is the asymptotic behavior of McKean-Vlasov SPDE with jumps?

However, up to now, the LDP for nonlinear McKean-Vlasov SPDEs with L\'{e}vy noise has not been considered in the previous works. The main aim of this work is to investigate the LDP for the following McKean-Vlasov SPDEs driven by Poisson random measure
\begin{equation}\label{YE00}
\left\{ \begin{aligned}	&dX^\varepsilon(t)=A(t,X^\varepsilon(t),\mathcal{L}_{X^\varepsilon(t)})dt
+ \varepsilon \int_Z  f(t,X^\varepsilon(t-),\mathcal{L}_{X^\varepsilon(t)},z)\widetilde{N}^{\varepsilon ^{-1}}(dt,dz),\\
	&X^\varepsilon(0)=x \in H,
\end{aligned} \right.
\end{equation}
where $\varepsilon >0$ is a small parameter, $\mathcal{L}_{X^\varepsilon(t)}$ stands for the distribution of solution $X^\varepsilon(t)$, $\widetilde{N}^{\varepsilon ^{-1}}$ is a compensated Poisson random measure on $[0,T] \times Z$ with a $\sigma$-finite intensity measure $\varepsilon^{-1} \lambda_T \otimes \theta,$ where $\lambda_T$ is the Lebesgue measure on $[0,T]$ and $\theta$ is a $\sigma$-finite measure on $Z.$ The precise conditions on the coefficients $A$ and $f$ are given in section \ref{PaM} below. In order to investigate the LDP for nonlinear
McKean-Vlasov SPDEs with jumps, we adapt the generalized variational framework in this work, since we want to apply it to some quasi-linear type McKean-Vlasov SPDEs with monotone coefficients. The classical variational framework has been established by Pardoux, Krylov and Rozovskii (see e.g. \cite{KR,RS}), where they employed the famous monotonicity tricks to verify the existence and uniqueness of solutions for SPDEs fulfilling the classical monotonicity and coercivity assumptions. We refer the interested readers to \cite{G82,HHL,LR2,LR,PR07} and reference therein for the recent development in such framework. However, due to the fact that the coefficients of equation \eref{YE00} depend on the distribution of the solution, there are no results in the literature regarding the existence and uniqueness of the solutions.

To this end, by using the fixed-point approach, we first establish the existence and uniqueness of strong solutions for McKean-Vlasov SPDE (\ref{YE00}). The well-posedness result is new in the literature, which generalises the classical results (cf. \cite{G82}) to the case of distribution dependent SPDEs.
Then we aim to prove the LDP for equation (\ref{YE00}) driven by small multiplicative L\'{e}vy noise, whose proof is mainly based on the weak convergence criteria established in
\cite{LSZZ}. The weak convergence approach for SDE with Poisson random measure was first introduced by Budhiraja et al. \cite{BCD,BCG16,BDM11} which has
became a very powerful tool to study the LDP for finite/infinite-dimensional SDEs driven by L\'{e}vy noise. Recently,  in \cite{LSZZ}, the authors successfully developed the weakly convergent method for studying the LDP of McKean-Vlasov SDEs with jumps, which is fundamentally different from the classical SDEs.
Instead of proving exponential probability estimates, one main advantage of using the weak convergence method is that, one only needs to establish some priori moment estimates, which significantly simplifies the proof.

Of course, in the finite dimensional case, the well-posedness and LDP for McKean-Vlasov SDEs driven by L\'{e}vy noise have been analyzed in the literature.
However, to the best of our knowledge, this is the first result concerning the well-posedness and LDP for McKean-Vlasov quasi-linear SPDEs with L\'{e}vy noise. Based on the generalized variational framework, our main results are applicable to several McKean-Vlasov quasi-linear SPDEs driven by Poisson random measure, such as distribution dependent stochastic porous media type equations and stochastic $p$-Laplace type equations (see Sect. \ref{App} below). Note that the existing results cannot be applied to this type of models, both the well-posedness and LDP established in this work are new in the literature.

In comparison
to the finite dimensional case, in order to deal with the infinite dimensional McKean-Vlasov SPDEs in the variational framework, we need to derive some apriori estimates of solutions involving different spaces, which is quite different to the finite dimensional case. Another point is that we want to drop the compactness assumption of Gelfand triple.
To overcome such difficulty, the time discretization procedure and relative entropy estimates mainly inspired by the recent work of Wu and Zhai \cite{WZ} (see also \cite{WZZ24}) are also employed, but some new strategies are needed to handle the difficulties caused by distribution dependence.

We mention that the LDP for standard SPDEs (i.e.~distribution independent case) driven by Poisson random measure has been studied a lot in the literature. R\"{o}ckner and Zhang \cite{RZ} first studied stochastic evolution
equations driven by additive L\'{e}vy noise by proving exponential probability estimates, where the Lipschitz coefficients are considered. Afterwards, the LDP for SPDEs with highly nonlinear terms has been been widely investigated in the literature. For example, by employing the weak convergence method, Brze\'{z}niak et al. \cite{BPZ} establish a Freidlin-Wentzell-type LDP for the solutions of 2D stochastic
Navier-Stokes equations with jumps. Xiong and Zhai \cite{XZ} first derived a LDP result for a class of SPDEs with locally monotone coefficients driven by Poisson random measure
(see \cite{CG,HLL21b,MO21,TM22,WZ20,XCV24,ZZ15} and references therein for recent progress on LDP for various SPDE models).

It is also worth noting that there is a significant  difference between McKean-Vlasov SPDEs and standard SPDEs in studying LDP.
For instance, as an important part of the proof, we need to formulate the correct form of the skeleton equation.
Intuitively, as the parameter $\varepsilon\rightarrow0$ in system (\ref{YE00}), the noise term vanishes, then we can get the following partial differential equation
\begin{equation*}
\frac{d X^0(t)}{d t} = A (t, X^0(t),\mathcal{L}_{ X^0(t)}), \quad X^0(0) = x,
\end{equation*}
where $\mathcal{L}_{{X}^0(t)}=\delta_{{X}^0(t)}$ is the Dirac measure of $X^0(t)$.  Then for any $g \in S$ (see \ref{e03} below), $t \in [0,T]$, we consider the following skeleton equation
\begin{align}\label{ske1}
X^g(t)=&x+\int_0^tA(s,X^g(s),\mathcal{L}_{X^0(s)})ds\nonumber\\
&+\int_0^t\int_{Z}f(s,X^g(s),\mathcal{L}_{X^0(s)},z)(g(s,z)-1)\theta(dz)ds.
\end{align}
 Note that $\mathcal{L}_{{X}^0(t)}$ (instead of  $\mathcal{L}_{X^g(t)}$) shows up in the skeleton equation (\ref{ske1}) and is used to define the rate function of LDP. Since the measure $\mathcal{L}_{X^{\varepsilon}(t)}$ is not random, the convergence of these
measures is independent of the occurrence of a rare event for the random variable ${X^{\varepsilon}(t)}$.

This paper is organized as follows.
In section 2, we will recall some basic notations and introduce our main results. Then, we provide several illustrating examples in section 3. In section 4, we aim to prove our main results. For simplicity, $C_p$ used in this paper will denotes positive constant which may change from line to line, where the subscript of $C_p$ emphasize the dependence on parameter $p$.

\section{Preliminaries and main results}\label{PaM}
\setcounter{equation}{0}
 \setcounter{definition}{0}

In this section, we will first give some notations and Gelfend triples in subsection \ref{SN}. Then we will introduce Poisson random measures in subsection \ref{PRM}. In subsection \ref{DaMR}, we show some relevant definitions and present main results of this paper.

\subsection{Some notations}\label{SN}
Fixing $T \in (0,\infty),$ let $\lambda_T$ denote the Lebesgue measure on $[0,T]$ and $\lambda_\infty$ denote the Lebesgue measure on $[0,\infty).$ We will denote the Borel $\sigma$-field on metric space $\mathbb{H}$ by $\mathcal{B}(\mathbb{H}).$ Let $L^\infty ([0,T] , \mathbb{H})$ be the space of all $\mathbb{H}$-valued uniformly bounded measurable functions on $[0,T],$ $C ([0,T] , \mathbb{H})$ be the space of all $\mathbb{H}$-valued continuous functions on $[0,T]$ equipped with the topology of uniform convergence, and $D ([0,T] , \mathbb{H})$ be the space of all $\mathbb{H}$-valued c\'{a}dl\'{a}g functions on $[0,T]$ equipped with the usual Skorohod topology.

Next we introduce the Gelfand triples used in this article. Let $(H , \langle \cdot,\cdot\rangle_{H})$ be a Hilbert space. The daul space of $H$ is denoted by $H^*.$  Let $(V ,\|\cdot\|_V)$ be a reflexive Banach space such that $V \subset H$ continuously and densely. Then for its daul space $V^*$ follows that $H^* \subset V^*$ continuously and densely. According to Riesz isomorphism theorem, we have that
$$V \subset H \subset V^* $$
is a Gelfand triple. If $_{V^*}\langle\cdot,\cdot\rangle_V$ denotes the dualization between $V^*$ and $V,$ then
\begin{equation}\label{e203}
_{V^*}\langle z,v \rangle_V = \langle z,v \rangle_H,\quad \text{for all }z \in H, v \in V.
\end{equation}

In order to solve the problem of distribution dependence in the proof process of this paper, we next introduce Wasserstein distance. Let $\mathcal{P}(H)$ denote the space of all probability measures with weak topology on $H.$ Furthermore, we set
$$\mathcal{P}_2(H):=\Big\{ \mu \in \mathcal{P}(H):\mu(\| \cdot \|_H^2):=\int_H \| \xi\| _H^2 \mu(d \xi)< \infty \Big\},$$
then $\mathcal{P}_2(H)$ is a Polish space under the following $L^2$-Wasserstein distance
$$\mathbb{W}_{2,H} (\mu, \nu):=\inf_{\pi \in \mathcal{C}(\mu, \nu)}\Big( \int_{H \times H} \| \xi - \eta\|_H^2 \pi(d\xi,d\eta)\Big)^{\frac{1}{2}} , \mu, \nu \in \mathcal{P}_2(H),$$
where $\mathcal{C}(\mu, \nu)$ is a set of the coupling of measure $\mu$ and $\nu.$ And $\pi \in \mathcal{C}(\mu, \nu)$ is the probability measure of $H \times H,$ i.e.
$$\pi(\cdot \times H)=\mu,\pi(H \times \cdot)=\nu.$$
For any $0\leq s <t < \infty,$ $C([s,t];\mathcal{P}_2(H))$ denotes a set of continuous mappings from $[s,t]$ to $\mathcal{P}_2(H)$ under the measure $\mathbb{W}_{2,H}.$

\subsection{Poisson random measure}\label{PRM}
Let $Z$ be a locally compact Polish space, denote by $\mathcal{M}_{FC}(Z)$ the space of all nonnegative measure $\nu$ on $(Z,\mathcal{B}(Z))$ such that $\nu(K)<\infty$ for every compact $K$ in $Z$. Let $C_c(Z)$ be the space of continuous functions with compact support. Endow $\mathcal{M}_{FC}(Z)$ with the weakest topology such that for every $h \in C_c(Z),$ the function $$\mathcal{M}_{FC}(Z) \ni \theta \mapsto \langle h, \nu \rangle := \int _Z h(u) \nu (du)$$ is continuous. Then $\mathcal{M}_{FC}(Z)$ is a Polish space under this topology.

Throughout this paper, we will use this topology on $\mathcal{M}_{FC}(Z)$, and let  $\theta$ be a given $\sigma$-finite positive measure on $(Z, \mathcal{B}(Z))$ with $\theta \in \mathcal{M}_{FC}(Z)$.
Define $Z_T := [0,T] \times Z,$  $Y := Z \times [0,\infty )$ and $Y_T :=[0,T] \times Y.$ In the following, we will write $\mathbb{M} := \mathcal{M}_{FC}(Y_T).$ Let $\mathbb{P}$ denote the unique probability measure on $(\mathbb{M},\mathcal{B}(\mathbb{M}))$ such that the canonical map $\bar{N}:\mathbb{M} \mapsto \mathbb{M},$ $\bar{N}(m):= m,$ is a Poisson random measure with intensity measure $\bar{\theta}_T = \lambda_T \otimes \theta \otimes \lambda_\infty .$ The corresponding compensated Poisson random measure is denoted by $\widetilde{\bar{N}}.$

Let $\bar{\mathcal{F}}_t:= \sigma\{\bar{N}((0,s]\times A):0\leq s\leq t, A \in \mathcal{B}(Y)\},$ and let $\mathcal{F}_t$ denote the completion under $\mathbb{P}.$ Denote the predictable $\sigma$-field on $[0,T] \times \mathbb{M}$ with the filtration $\{\mathcal{F}_t:0 \leq t \leq T \}$ on $(\mathbb{M},\mathcal{B}(\mathbb{M}))$ by $\bar{\mathcal{P}}.$  Let $\mathcal{A}$ be the class of all $(\bar{\mathcal{P}} \otimes \mathcal{B}(Z)) / [0,\infty)$-measurable maps $\varphi : Z_T \times \mathbb{M} \to [0, \infty).$ For $\varphi \in \mathcal{A},$ define a counting process $N^\varphi$ on $Z_T$ by
\begin{equation}\label{e201}
N^\varphi((0,t]\times U) = \int _{(0,t]\times U \times (0,\infty)}1_{[0,\varphi(s,z)]}(r)\bar{N}(ds,dz,dr), \text{ for }t \in [0,T] \text{ and } U \in \mathcal{B}(Z) .
\end{equation}
$N^\varphi$ is called a controlled random measure, with $\varphi$ selecting the intensity for the points at location $z$ and times $s,$ in a possibly random but non-anticipating way. Analogously, we define a process
\begin{equation}\label{e202}
\widetilde{N}^\varphi((0,t]\times U) = \int _{(0,t]\times U\times (0,\infty)} 1_{[0,\varphi(s,z)]}(r)\widetilde{\bar{N}}(ds,dz,dr).
\end{equation}
When $\varphi (s,z) \equiv \varepsilon^{-1} \in (0,\infty),$ we write $N^\varphi = N^{\varepsilon^{-1}}$ and $\widetilde{N}^\varphi = \widetilde{N}^{\varepsilon ^{-1}}.$ Set $\theta_T:= \lambda_T \otimes \theta.$ Note that, with respect to $\mathbb{P},$ $N^{\varepsilon^{-1}}$ is a Poisson random measure on $Z_T$ with intensity measure $\varepsilon^{-1}\theta_T,$ and $\widetilde{N}^{\varepsilon ^{-1}}$ is the compensated Poisson random measure. In particular, we donate it by $\widetilde{N}$, when $\varepsilon\equiv1$.

Set $(\Omega,\mathcal{F})=(\mathbb{M},\mathcal{B}(\mathbb{M})).$ In the present paper, we study (\ref{YE00}) on the given probability space $(\Omega,\mathcal{F},\{\mathcal{F}_t,t \in [0,T]\},\mathbb{P}).$ Denote the expectation with respect to $\mathbb{P}$ by $\mathbb{E}.$

For $\varpi \in (0,\infty),$ define
\begin{align*}
\mathcal{H}^{\varpi}:=
&\Big\{ h:[0,T] \times Z \rightarrow \mathbb{R}^{+} : \forall
~\Gamma \in \mathcal{B}([0,T]) \otimes \mathcal{B}(Z),\\
&\text{with }\theta_T(\Gamma)< \infty, \text{we have }\int_{\Gamma} \exp ( \varpi h(s,z))\theta(dz)ds< \infty\Big\},
\end{align*}	
and denote by $\mathcal{H}^{\infty} = \bigcap_{\varpi \in (0,\infty)} \mathcal{H}^{\varpi}.$
Set
\begin{align*}
\mathcal{H}_2:=
&\Big\{ h:[0,T] \times Z \rightarrow \mathbb{R}^{+} : \exists~\delta > 0,s.t. ~\forall~\Gamma \in \mathcal{B}([0,T]) \otimes \mathcal{B}(Z),\\
&\text{with }\theta_T(\Gamma)< \infty, \text{we have }\int_{\Gamma} \exp ( \delta h^2(s,z))\theta(dz)ds< \infty\Big\},
\end{align*}	
then by \cite[Remark 3.2]{BCD}, we have
$$
\mathcal{H}_2 \subset \mathcal{H}^{\infty} .
$$
Denote
$$
L_2\left(\theta_T\right)=\Big\{h:[0, T] \times Z \rightarrow \mathbb{R}^{+}: \int_0^t \int_Z h^2(s, z) \theta(dz)ds<\infty\Big\}. \\
$$

\subsection{Main results}\label{DaMR}

Let $T > 0$ be fixed. For some measurable maps
\begin{align*}
& A:[0, T] \times V \times \mathcal{P}_2(H) \rightarrow V^*, \\
& f:[0, T] \times V \times \mathcal{P}_2(H) \times Z \rightarrow H ,
\end{align*}
we consider the following type of McKean-Vlasov stochastic evolution equation,
\begin{equation}\label{YE01}
\left\{ \begin{aligned}	&dX(t)=A(t,X(t),\mathcal{L}_{X(t)})dt
+  \int_Z  f(t,X(t-),\mathcal{L}_{X(t)},z)\widetilde{N}^{}(dt,dz),\\
	&X(0) = X_0.
\end{aligned} \right.
\end{equation}

In this paper, we impose that $A$ and $f$ satisfy the following assumptions.
\begin{hypothesis}\label{h1}
We suppose that there exists constants $\alpha>1,c \in \mathbb{R},$ and $\delta>0$ such that
\begin{enumerate}
\item [$(H1)$]For any $t \in [0,T] ,$ mapping
$$
V \times \mathcal{P}_{2}(H) \ni (x,\mu) \mapsto _{V^*}\langle A(t,x,\mu),y\rangle_{V}
$$
is continuous.

\item [$(H2)$]For any $x \in V,$ $\mu \in \mathcal{P}_2 (H)$ and $t \in [0,T]$,
$$
\begin{aligned}
&2 _{V^*}\langle A(t,x,\mu),x\rangle_{V}
\leq c (\|x\|_H^2 + \mu(\|\cdot\|_H^2) +1)- \delta \|x\|_V^\alpha .
\end{aligned}
$$

\item [$(H3)$]For any $x,y \in V$, $\mu ,\nu \in \mathcal{P}_2 (H)$ and $t \in [0,T]$,
$$
\begin{aligned}
2\,_{V^*}\langle A(t,x,\mu)-A(t,y,\nu),x-y\rangle_{V}
\leq c\|x-y\|_{H}^{2}+c \mathbb{W}_{2,H} (\mu,\nu)^2.
\end{aligned}
$$

\item [$(H4)$]For any $x \in V,$ $\mu \in \mathcal{P}_2 (H)$, $t \in [0,T]$ and $\alpha>1$,
$$
\|A(t,x,\mu)\|_{V^*}^{\frac{\alpha}{\alpha-1}}\leq c (\|x\|_V^\alpha +  \mu(\|\cdot\|_H^2) +1).
$$

\item [$(H5)$]There exist $l_1 \in \mathcal{H}^{\infty} \cap L_2(\theta_T)$ and $l_2 \in \mathcal{H}_2 \cap L_2(\theta_T)$ such that for all $(t,z)\in[0,T] \times Z,x,y\in V$,

  \item [$(i)$]  $$\|f(t,x,\mu,z)-f(t,y,\nu,z)\|_H \leq l_1 (t,z)(\|x-y\|_H+\mathbb{W}_{2,H} (\mu,\nu)),$$

  \item [$(ii)$] $$\|f(t,x,\mu,z)\|_H\leq l_2 (t,z)(\|x\|_H + \mu(\|\cdot\|_H^2)^\frac{1}{2}+1).$$
\end{enumerate}
\end{hypothesis}
\begin{Rem}
Hypothesis \ref{h1} is a fairly standard assumption when one considers the
existence and uniqueness of solutions to SPDEs under the classical variational framework, see e.g. \cite{KR,PR07,RS} for the Gaussian noise and \cite{BLZ,G82} for the jump noise.
 \end{Rem}

The first main result in this work is about the existence and uniqueness of solutions to Eq.~(\ref{YE01}).
\begin{theorem}\label{THY}
Suppose that $(H1)$-$(H5)$ hold. Then for any initial value $X_0\in L^2(\Omega; H)$, there exists a unique strong solution $X$ to Eq.~$(\ref{YE01}),$ i.e., $X=(X(t), t \in [0,T])$ is an $H$-valued c\'{a}dl\'{a}g $\mathcal{F}_t$-adapted process, and the following conditions are satisfied
$$X \in L^{\alpha}([0,T]\times \Omega;V))\cap L^2( \Omega;L^\infty([0,T];H)),$$
for any $t \in [0,T]$,
$$X(t)= X_0 +\int_0^tA(s,X(s),\mathcal{L}_{X(s)})ds+ \int_0^t\int_{Z}f(s,X(s-),\mathcal{L}_{X(s)},z)\widetilde{N}(ds,dz),$$
holds $\mathbb{P}$-a.s..
\end{theorem}

The proof of Theorem \ref{THY} is given in Section \ref{PT2.1}.

\begin{Rem}
To the best of our knowledge, this is the first result concerning the well-posedness for McKean-Vlasov quasi-linear SPDEs with L\'{e}vy noise, which extends the classical well-posedness results in the variational framework (cf. \cite{G82}) to the case of distribution dependent SPDEs, some examples will be given in Section \ref{App} below.
 \end{Rem}

The second purpose of this paper is to establish a LDP for McKean-Vlasov SPDE (\ref{YE01}) with small L\'{e}vy noise, i.e., considering the following equation
\begin{equation}\label{YE}
\left\{ \begin{aligned}	&dX^\varepsilon(t)=A(t,X^\varepsilon(t),\mathcal{L}_{X^\varepsilon(t)})dt
+ \varepsilon \int_Z  f(t,X^\varepsilon(t-),\mathcal{L}_{X^\varepsilon(t)},z)\widetilde{N}^{\varepsilon ^{-1}}(dt,dz),\\
	&X^\varepsilon(0)=x\in H,
\end{aligned} \right.
\end{equation}
where $\varepsilon >0$ is a small parameter, we are interested in the asymptotic behavior of
$X^\varepsilon$ on $D([0,T];H)$ as $\varepsilon \to 0.$

Before the statement of the main result, we introduce the definition of LDP. Let $\{X^\varepsilon\}_{\varepsilon > 0}$ denote a family of random variables defined on a probability space $(\Omega, \mathcal{F},\mathbb{P})$ and taking values in a Polish space $E.$ The theory of large deviations is concerned with events $\mathcal{A} \in \mathcal{B}(E)$ for which probability $\mathbb{P} (X^\varepsilon\in \mathcal{A})$ converges to zero exponentially fast as $\varepsilon \to 0.$ The exponential decay rate such probabilities is typically expressed in terms of a ``rate function" $I$ defined bellow.

\begin{definition}
A function $I:E\rightarrow[0,+\infty]$ is called a rate function on $E,$ if $I$ is lower semi-continuous. Moreover, a rate function $I$
is called a {\it good rate function} if  the level set $\{x\in E: I(x)\le
K\}$ is compact for each constant $K<\infty.$
\end{definition}

\begin{definition} The random variable family
 $\{X^\varepsilon\}$ is said to satisfy
 the LDP on $E$ with rate function
 $I$ if  the following lower and upper bound conditions hold.

(i) (Lower bound) for any open set $G\subset E,$ then
$$\liminf_{\varepsilon\to 0}
   \varepsilon \log \mathbb{P}(X^{\varepsilon}\in G)\geq -\inf_{x\in G}I(x),$$

(ii) (Upper bound) for any closed set $F\subset E,$ then
$$ \limsup_{\varepsilon\to 0}
   \varepsilon \log \mathbb{P}(X^{\varepsilon}\in F)\leq
  -\inf_{x\in F} I(x).
$$
\end{definition}

Before giving the LDP result, we need to introduce the skeleton equation, which is used to define the rate function. To do this, some more basic notations should be introduced. Define $l:[0,\infty)\to [0,\infty)$ by
\begin{equation}\label{e01}
l(r)=r\log r-r+1,r\in [0,\infty).
\end{equation}
For any $\varphi\in \mathcal{A}$,
\begin{equation}\label{e02}
Q(\varphi)=\int_{Z_T}l(\varphi(t,z))\theta_T(dzdt)
\end{equation}
is well defined as a $[0,\infty]$-valued random variable.

Let $N \in \mathbb{N},$ define
\begin{equation}\label{e03}
S=\bigcup_{N=1}^\infty S^N,~~\text{where}~~S^N= \Big\{g:Z_T \to [0,\infty):Q(g)\leq N \Big\} .
\end{equation}
A measure $\theta^g_T \in \mathcal{M}_{FC}(Z_T)$ can identify a function $g \in S^N,$ defined by
\begin{equation}\label{e04}
\theta^g_T (A)= \int _A g(s,z)\theta_T(dzds),A \in \mathcal{B}(Z_T).
\end{equation}
This identification introduces a topology on $S^N$ under which $S^N$ is a compact space (\cite[Appendix]{BCD}). Throughout the paper, we use this topology on $S^N.$

For any $g \in S,$ $t \in [0,T]$, we consider the following skeleton equation
\begin{align}\label{e05}
X^g(t)=&x+\int_0^tA(s,X^g(s),\mathcal{L}_{X^0(s)})ds\nonumber\\
&+\int_0^t\int_{Z}f(s,X^g(s),\mathcal{L}_{X^0(s)},z)(g(s,z)-1)\theta(dz)ds,
\end{align}
where $X^0$ is the solution of the following equation
\begin{equation}\label{e306}
\frac{d X^0(t)}{d t} = A (t, X^0(t),\mathcal{L}_{ X^0(t)}), \quad X^0(0) = x.
\end{equation}
The reasons for considering such a skeleton equation will be given in the next subsection.

The second main result of this paper concerning the LDP for equation (\ref{YE}) is as follows.

\begin{theorem}\label{LDP}
 Suppose that $(H1)$-$(H5)$ hold. Then the solution of (\ref{YE}), i.e., $X^\varepsilon$ satisfies a LDP on $D([0,T];H)$ with the rate function
$$I(\phi)=\inf\Big\{ Q(g):\phi=X^g,g \in S\Big\},\quad\phi \in D([0,T];H),$$
where $X^g$ is the solution to the skeleton equation (\ref{e05}), $Q$ is defined in (\ref{e02}), and $\inf \varnothing = \infty $.
\end{theorem}

The proof of Theorem \ref{LDP} will be given in Section \ref{PT2.2}.
\begin{Rem}
 Theorem \ref{LDP} is the first result concerning the LDP for McKean-Vlasov quasi-linear SPDEs with L\'{e}vy noise, some examples will be given in Section \ref{App} below.
 \end{Rem}

\section{Applications and examples}\label{App}

The main results formulated in Theorems \ref{THY} and \ref{LDP} can be used to deal with a class of Mckean-Vlasov SPDE models driven by L\'{e}vy noise. Here, we illustrate our main results by the stochastic porous media equation and stochastic $p$-Lapace equation.

As preparations, we first recall some definitions and notations which are commonly used in the literature. Let $\Lambda \subset \mathbb{R}^d$ is a bounded domain with smooth boundary $\partial \Lambda$, $C_0^\infty (\Lambda, \mathbb{R}^d)$ be the space of all smooth function from $\Lambda$ to $\mathbb{R}^d$ with compact support. For any $r \geq 1,$ let $L^r (\Lambda, \mathbb{R}^d)$ be the vector valued $L^r$-space with the norm $\|\cdot\|_{L^r}.$ For any integer $m>0,$ we denote the classical Sobolev space by $W_0^{m,r} (\Lambda, \mathbb{R}^d)$, which satisfies Dirichlet boundary condition and is equipped with the following norm
$$
\|u\|_{W^{m,r}} = \Bigg(\sum_{|\alpha |=m}{\int_{\Lambda}{|D^{\alpha}u|^r}}dx \Bigg)^{\frac{1}{r}} .
$$

\subsection{McKean-Vlasov stochastic porous media equation}

We consider the following Mckean-Vlasov porous media equation
\begin{equation}\label{e601}
\left\{ \begin{array}{l}
	dX^{\varepsilon}( t ) =\Delta \Psi (t, X^{\varepsilon}( t ) ,\mathcal{L}_{X^{\varepsilon}( t )} ) +\varepsilon \int_Z{f(t, X^{\varepsilon}( t- ) ,\mathcal{L}_{X^{\varepsilon}( t )},z )}\widetilde{N}^{\varepsilon ^{-1}}( dt,dz ) , \\
	X^{\varepsilon}( 0 ) =x ,
\end{array} \right.
\end{equation}
where $\Delta$ denotes the Laplace operator, and $\Psi,$ $f$ satisfy some assumptions below.

For any $r \geq 2,$ we set the following Gelfand triple for (\ref{e601})
$$
V:= L^r (\Lambda) \subset H := (W_0^{1,2}(\Lambda))^* \subset V^*.
$$

We recall the following useful lemma (refer Lemma 4.1.13 in \cite{LR}).

\begin{lemma}\label{l601}
The map
$$\Delta: W_0^{1,2} (\Lambda) \to (L^r (\Lambda))^* $$
could be extend to a linear isometry
$$\Delta: L^{\frac{r}{r-1}}(\Lambda) \to (L^r (\Lambda))^*.$$
Furthermore, for any $u \in L^{\frac{r}{r-1}}(\Lambda),$ $v \in L^r (\Lambda)$ we have
$$
_{V^*} \langle -\Delta u, v \rangle_V = _{L^{\frac{r}{r-1}}} \langle u, v \rangle _{L^r} = \int _\Lambda u(\xi)v(\xi)d \xi.
$$
\end{lemma}

Firstly, we suppose the measurable map $\Psi: V \times \mathcal{P}_2(H) \to  L^{\frac{r}{r-1}}(\Lambda)$ satisfies the following hypotheses.
\begin{hypothesis}\label{h3}For all $u,v \in V$ and $\mu, \nu \in \mathcal{P}_2(H)$
\begin{enumerate}
\item [$(\Psi 1)$] The map
$$
V \times \mathcal{P}_2(H) \ni (u,\mu) \mapsto \int_\Lambda \Psi(u,\mu)(\xi) v (\xi)d \xi
$$
is continuous.

\item [$(\Psi 2)$] There are some constants $c$, $\delta > 0$ such that
$$
\int_\Lambda \Psi(t,u,\mu)(\xi) v (\xi)d \xi \geq -c(\|u\|_H ^2+ \mu (\|\cdot\|_H^2)+1)+\delta\|u\|_V^r.
$$

\item [$(\Psi 3)$]
$$
\int_\Lambda (\Psi(t,u,\mu)(\xi)-\Psi(t,v,\nu)(\xi))(u(\xi)-v(\xi))d \xi \geq 0.
$$

\item [$(\Psi 4)$]
There is a constant $c>0$
$$
\|\Psi(t,u,\mu)\|^{\frac{r}{r-1}}_{L^{\frac{r}{r-1}}} \leq c(\|u\|_V^r + \mu (\|\cdot\|_H^2)+1).
$$
\end{enumerate}
\end{hypothesis}

After the preparations above, we now define map $A: V \times \mathcal{P}_2(H) \to V^*$ by
$$
A(t,u,\mu):= \Delta \Psi (t,u, \mu).
$$
The Lemma \ref{l601} ensures that the map $A$ is well-defined and takes value in $V^*$. Moreover, it is easy to check that the conditions $(\Psi 1)$-$(\Psi 4)$ imply $(H 1)$-$(H 4)$ hold with $\alpha=r$. In order to prove the main result, we further assume that there exists $l_1,l_2 \in \mathcal{P}_2 \cap L_2(\theta_T)$ such that for all $(t,z) \in [0,T] \times Z,u,v \in H,$ $f$ satisfies $(H5).$

According to Theorems \ref{THY} and \ref{LDP}, we have the following result for the distribution dependent stochastic porous media equation.

\begin{theorem}
Suppose that $\Psi$ satisfies the conditions $(\Psi 1)$-$(\Psi 4)$ above and assumption $(H5)$ holds. Then for any initial value $x \in H$ and $T>0$, Eq.~(\ref{e601}) has a unique solution $X^\varepsilon(t),t \in [0,T]$, and $X^\varepsilon(t)$ satisfies the LDP on $D([0,T];H)$ with the rate function
$$I(\phi)=\inf\{ Q(g):\phi=X^g,g \in S\},\quad \phi \in D([0,T];H),$$
where $X^g$ is the solution of the corresponding skeleton equation.
\end{theorem}

\begin{Rem}
To the best of our knowledge, there is no result in the literature concerning the well-posedness and LDP
for McKean-Vlasov stochastic porous media equation driven by L\'{e}vy type noise. Recently, Hong et al. \cite{HHL,HLL21} have studied McKean-Vlasov stochastic porous media equation driven by Gaussian noise.
\end{Rem}

\subsection{McKean-Vlasov stochastic $p$-Lapace equation}
Now we apply our main result to the following Mckean-Vlasov stochastic $p$-Lapace equation
\begin{equation}\label{e606}
\left\{ \begin{array}{l}
	dX^{\varepsilon}( t ) =div( |\nabla X^{\varepsilon}( t ) |^{p-2}\nabla X^{\varepsilon}( t ) ) dt
+K(t,X^{\varepsilon}( t ) ,\mathcal{L}_{X^{\varepsilon}( t )})dt \\
~~~~~~~~~~~~+\varepsilon \int_Z{f(t, X^{\varepsilon}( t- ) ,\mathcal{L}_{X^{\varepsilon}( t )},z )}\widetilde{N}^{\varepsilon ^{-1}}( dt,dz ), \\
	X^{\varepsilon}( 0 ) =x .\\
\end{array} \right.
\end{equation}
For any $p \geq 2,$ we set the following Gelfand triple for (\ref{e606})
$$
V:= W_0^{1,p}(\Lambda) \subset H:= L^2(\Lambda) \subset V^*.
$$

Let $\bar{A}(u) := div(|\nabla u|^{p-2}\nabla u),$ which is called $p$-$Laplacian$ operator. It is well known that the operator $\bar{A}$ satisfies $(H1)$-$(H4)$ with $\alpha=p$ (cf. \cite{LR}).
Suppose that for all $t\in[0,T]$, $u,v\in V$ and $\mu,\nu\in\mathcal{P}_2(H)$, there is a constant $C>0$ such that the map $K:[0,T]\times V\times\mathcal{P}_2(H)\to H$ satisfies
\begin{eqnarray}\label{31}
\|K(t,u,\mu)-K(t,v,\nu)\|_H\leq C\big(\|u-v\|_H+\mathbb{W}_{2,H}(\mu,\nu)\big).
\end{eqnarray}

According to Theorems \ref{THY} and \ref{LDP}, we get the following result for the distribution dependent stochastic $p$-Laplace equation with L\'{e}vy noise.

\begin{theorem}
Suppose that $(H5)$ hold, then for any initial value $x \in H$ and $T>0,$ system (\ref{e606}) has a unique solution $X^\varepsilon(t),t \in [0,T]$, and $X^\varepsilon(t)$ satisfies the LDP on $D([0,T];H)$ with the rate function
$$I(\phi)=\inf\{ Q(g):\phi=X^g,g \in S\},\quad \phi \in D([0,T];H),$$
where $X^g$ is the solution of the corresponding skeleton equation.
\end{theorem}
\begin{Rem} Both the well-posedness and LDP for Mckean-Vlasov stochastic $p$-Lapace equation driven by L\'{e}vy noise are new in the literature. In particular, if we take $p=2$, $\bar{A}$ reduces to the classical Laplace operator. Therefore, our result above also covers a class of distribution dependent semilinear SPDEs.
\end{Rem}

\subsection{McKean-Vlasov SDEs}
Our main results are applicable to McKean-Vlasov SDE models. For instance, we consider $V = H = \mathbb{R}^d$ with the Euclidean norm $|\cdot|$ and  inner product $\langle \cdot,\cdot \rangle,$
\begin{equation}\label{e5501}
dX^\varepsilon(t)= b(t,X^\varepsilon(t),\mathcal{L}_{X^\varepsilon(t)})dt + \varepsilon \int_Z{f(t, X^{\varepsilon}( t- ) ,\mathcal{L}_{X^{\varepsilon}( t )},z )}\widetilde{N}^{\varepsilon ^{-1}}( dt,dz ),\quad X^\varepsilon(0)=x.
\end{equation}
Suppose the coefficients are measurable and satisfy the following conditions.
\begin{hypothesis}\label{h2}
For all $t\in[0,T]$, $u,v\in \mathbb{R}^d$ and $\mu,\nu\in\mathcal{P}_2(\mathbb{R}^d)$,
\begin{enumerate}
\item [$({\mathbf{A}}{\mathbf{1}})$] (Continuity) The map
\begin{eqnarray*}
\mathbb{R}^d\times\mathcal{P}_2(\mathbb{R}^d)\ni(u,\mu)\mapsto b(t,u,\mu)
\end{eqnarray*}
is continuous.
\item [$({\mathbf{A}}{\mathbf{2}})$] (Monotonicity) There exists a constant $C>0$ such that
\begin{eqnarray*}
\langle b(t,u,\mu)-b(t,v,\nu),u-v\rangle \leq C\big(|u-v|^2+\mathbb{W}_{2,\mathbb{R}^d}(\mu,\nu)^2\big).
\end{eqnarray*}
\item [$({\mathbf{A}}{\mathbf{3}})$] (Growth) There is a constant $C>0$ such that
\begin{eqnarray*}
|b(t,u,\mu)|\leq C\big(1+|u|+\mu(|\cdot|^2)^\frac{1}{2}\big).
\end{eqnarray*}
\item [$({\mathbf{A}}{\mathbf{4}})$]
 There exists a constant $C>0$ such that
\begin{equation*}
\|f(t,x,\mu,z)-f(t,y,\nu,z)\| \leq C(\|x-y\|+\mathbb{W}_{2,\mathbb{R}^d} (\mu,\nu)),
\end{equation*}
\begin{equation*}
\|f(t,x,\mu,z)\|\leq C(\|x\|+ \mu(\|\cdot\|^2)^\frac{1}{2}+1),
\end{equation*}
where $\|\cdot\|$ denotes the matrix norm.
\end{enumerate}
\end{hypothesis}

By Theorem \ref{LDP}, we can derive the well-posedness and LDP for the McKean-Vlasov SDEs.

\begin{theorem}
Suppose that Hypothesis \ref{h2} hold, then for any initial value $x \in \mathbb{R}^d$ and $T>0,$ Eq.~(\ref{e5501}) has a unique solution $X^\varepsilon(t), t \in [0,T]$, and $X^\varepsilon(t)$ satisfies the LDP on $D([0,T];\mathbb{R}^d)$ with the rate function
$$I(\phi)=\inf\{ Q(g):\phi=X^g,g \in S\},\quad \phi \in D([0,T];H),$$
where $X^g$ is the solution of the corresponding skeleton equation.
\end{theorem}

\begin{Rem}
In \cite{LSZZ}, the authors first established LDP for McKean-Vlasov SDEs driven by L\'{e}vy noise. In this work, by applying our main result, we give a more succinct proof for the LDP associated with some McKean-Vlasov SDEs under monotonicity
and linear growth assumptions.
\end{Rem}

\section{Proof of main results}\label{mainpr}

\subsection{Proof of Theorem \ref{THY}}\label{PT2.1}
\setcounter{equation}{0}
 \setcounter{definition}{0}

The proof Theorem \ref{THY} is divided into the following two steps.

\medskip
\textbf{Step 1}: For any $0 \leq s < t \leq T,$ $\mu (\cdot) \in C([s,T]; \mathcal{P}_2(H))$ and $\psi \in \mathcal{P}_2(H),$ we consider the following SPDE with initial distribution $\mathcal{L}_{X_{s,s}^{\psi,\mu}}=\psi$

\begin{equation}\label{e701}
dX_{s,t}^{\psi,\mu}=A^{\mu}(t,X_{s,t}^{\psi,\mu})dt+ \int_{Z}f^{\mu}(t,X_{s,t-}^{\psi,\mu},z)
\widetilde{N}(dt,dz), t \in [s,T],
\end{equation}
where $A^{\mu}(t,x):= A(t,x,\mu)$ and $f^{\mu}(t,x,z):= f(t,x,\mu,z)$. According to \cite[Theorem 1.2]{BLZ}, equation (\ref{e701}) has a unique solution with initial distribution in $\mathcal{P}_2(H).$ The solution $\{X_{s,t}^{\psi,\mu}\}_{t \in [s,T]}$ is a $(\mathcal{F}_t)_{t \geq 0} $ adapt process, which satisfies $\mathcal{L}_{X_{s,t}^{\psi,\mu}} \in C([s,T]; \mathcal{P}_2(H))$ and takes values in $H,$ meanwhile
$$
\mathbb{E}\Big[\sup_{t \in [s,T]} \|X_{s,t}^{\psi,\mu}\|_H^2\Big]< \infty.
$$

For the solution $\{X_{s,t}^{\psi,\mu}\}_{t \in [0,T]}$ to (\ref{e701}), we define the following map
$$
\Phi_{s,\cdot}^\psi : C([s,T]; \mathcal{P}_2(H)) \to C([s,T]; \mathcal{P}_2(H)),
$$
by
$$
\Phi_{s,t}^\psi (\mu):= \mathcal{L}_{X_{s,t}^{\psi,\mu}},\quad t \in [s,T],~\mu \in C([s,T]; \mathcal{P}_2(H)).
$$
Note that $(X^\psi,\mu)$ is the solution to the following distribution dependent SPDE with initial distribution $\psi$
\begin{equation}\label{e702}
dX(t)=A(t,X(t),\mathcal{L}_{X(t)})dt+ \int_{Z}f(t,X(t-),\mathcal{L}_{X(t)},z)\widetilde{N}(dt,dz)
\end{equation}
if and only if $X_{s,t}^{\psi}=X_{s,t}^{\psi,\mu},\mu (t)= \Phi _{s,t}^{\psi} (\mu), t \in [s,T].$
More precisely, the fixed point of map $\Phi _{s,\cdot}^{\psi}$ is the solution to (\ref{e702}).

Next we prove the invariant compressibility of
$\Phi _{s,\cdot}^{\psi}$ for the following complete metric
$$
d_t (\mu, \nu):= \sup_{r \in [s,t]} e^{- \lambda r} \mathbb{W}_{2,H}(\mu(r),\nu(r)),
$$
where, for any $0 \leq s<t \leq T, \mu,\nu \in C([s,T]; \mathcal{P}_2(H)),$ constant $\lambda > 0$ is in the subspace $M_t:=\{\mu \in C([s,T]; \mathcal{P}_2(H)): \mu (s) = \psi \}.$

Let $\mu,\nu \in C([s,T]; \mathcal{P}_2(H)),X_{s,s}^\psi$ is a $\mathcal{F}_s$-measurable random variable and $\mathcal{L}_{X_{s,s}^{\psi}}=\psi.$  For any $t \in [s,T]$, we consider the following SPDEs,
$$
d X_{s, t}^{\psi, \mu}=A^\mu(t, X_{s, t}^{\psi, \mu}) d t + \int_0^t\int_{Z} f^\mu(t, X_{s, t-}^{\psi, \mu}, z) \widetilde{N}(d t, d z), X_{s, s}^{\psi, \mu}=X_{s, s}^\psi,
$$
$$
d X_{s, t}^{\psi, \nu}=A^\nu(t, X_{s, t}^{\psi, \nu}) d t +\int_0^t\int_{Z} f^\nu(t, X_{s, t-}^{\psi, \nu}, z) \widetilde{N}(d t, d z), X_{s, s}^{\psi, \nu}=X_{s, s}^\psi .
$$
For $\|\cdot\|_H^2,$ we apply It\^{o}'s formula
\begin{align*}
& \|X_{s, t}^{\psi, \mu}-X_{s, t}^{\psi, \nu}\|_H^2 \nonumber\\
= & \int_s^t\Big[2_{V^*}\langle A(r, X_{s, r}^{\psi, \mu}, \mu(r))-A(r, X_{s, r}^{\psi, \nu}, \nu(r)), X_{s, r}^{\psi, \mu}-X_{s, r}^{\psi, \nu}\rangle_V\Big] d r \nonumber\\
& +2 \int_s^t \int_{Z}\langle X_{s, r-}^{\psi, \mu}-X_{s, r-}^{\psi, \nu} , f(r, X_{s, r-}^{\psi, \mu}, \mu(r), z)-f(r, X_{s, r-}^{\psi, \nu}, \nu(r),z)\rangle_H \widetilde{N}(d r, d z)\nonumber\\
& +\int_s^t \int_{Z}\Big[\|(X_{s, r-}^{\psi, \mu}+f(r, X_{s, r-}^{\psi, \mu}, \mu(r), z))-(X_{s, r-}^{\psi, \nu}+f(r, X_{s, r-}^{\psi, \nu}, \nu(r), z))\|_H^2\nonumber\\
& -(\|X_{s, r-}^{\psi, \mu}\|_H^2-\|X_{s, r-}^{\psi, \nu}\|_H^2) \nonumber\\
& -2\langle X_{s, r-}^{\psi, \mu}-X_{s, r-}^{\psi, \nu},(f(r, X_{s, r-}^{\psi, \mu},\mu(r), z)-f(r, X_{s, r-}^{\psi, \nu},\nu(r), z)\rangle_H\Big] N(d r, d z) .
\end{align*}

Note that
\begin{equation}\label{e703}
\left|\|x+h\|_H^2-\|x\|_H^2-2\langle x, h\rangle_H\right|=\|h\|_H^2 .
\end{equation}

According to $(H5)(ii)$, we have
\begin{equation}\label{e704}
\int_Z \|f^{\mu}(t,v,z)\|_H^2\theta(dz) \leq C( 1 +\|v\|_H^2 + \mu(\|v\|_H^2) ).
\end{equation}

From (\ref{e703}) and (\ref{e704}) we have
\begin{align}\label{e705}
& \mathbb{E} \Bigg[ \int_s^t \int_{Z}\Big(\|X_{s, r-}^{\psi, \mu}+f(r, X_{s, r-}^{\psi, \mu}, \mu(r), z)\|_H^2-\|X_{s, r-}^{\psi, \mu}\|_H^2 \nonumber\\
& - 2\langle X_{s, r-}^{\psi, \mu}, f(r, X_{s, r-}^{\psi, \mu}, \mu(r), z)\rangle_H\Big) N(d r, d z)\Bigg] \nonumber\\
= & \mathbb{E} \Bigg[\int_s^t \int_{Z}\Big[\|X_{s, r}^{\psi, \mu}+f(r, X_{s, r}^{\psi, \mu}, \mu(r), z)\|_H^2-\|X_{s, r}^{\psi, \mu}\|_H^2\nonumber\\
& -2\langle X_{s, r}^{\psi, \mu}, f(r, X_{s, r}^{\psi, \mu}, \mu(r), z)\rangle_H\Big] \theta(d z) d r \Bigg] \nonumber\\
= & \mathbb{E} \Bigg[\int_s^t \int_{Z}\|f(r, X_{s, r}^{\psi, \mu}, \mu(r), z)\|_H^2 \theta(d z) d r \Bigg] \nonumber\\
\leq & C(|t-s|+ \|X_{s, t}^{\psi, \mu}\|_H^2+ \mu(\|X_{s, t}^{\psi, \mu}\|_H^2) ).
\end{align}

By $(H3)$ and (\ref{e705}), for any $ \kappa > 0$
\begin{align*}
& e^{-\kappa t} \mathbb{E}\|X_{s, t}^{\psi, \mu}-X_{s, t}^{\psi, \nu}\|_H^2 \\
= & \int_s^t e^{-\kappa r} d(\mathbb{E}\|X_{s, r}^{\psi, \mu}-X_{s, r}^{\psi, \nu}\|_H^2)+\int_s^t \mathbb{E}\|X_{s, r}^{\psi, \mu}-X_{s, r}^{\psi, \nu}\|_H^2 d e^{-\kappa r} \\
\leq & -\kappa \int_s^t e^{-\kappa r} \mathbb{E}\|X_{s, r}^{\psi, \mu}-X_{s, r}^{\psi, \nu}\|_H^2 d r+2 c \int_s^t e^{-\kappa r} \mathbb{E}[\|X_{s, r}^{\psi, \mu}-X_{s, r}^{\psi, \nu}\|_H^2 d r \\
& +c \int_s^t \mathbb{W}_{2, H}(\mu(r), \nu(r))^2 d r .
\end{align*}
Let $\kappa = 2c,$ then
\begin{equation}\label{e706}
e^{-2 c t} \mathbb{E}\|X_{s, t}^{\psi, \mu}-X_{s, t}^{\psi, \nu}\|_H^2 \leq(2 c) \int_s^t e^{-c r} \mathbb{W}_{2, H}(\mu(r), \nu(r))^2 d r.
\end{equation}
We take both sides of (\ref{e706}) to maximizes. Note that the joint distribution form of $(X_{s, \cdot}^{\psi, \mu},X_{s,\cdot}^{\psi, \nu})$ is the coupling of $(\Phi_{s, \cdot}^{\psi}(\mu),\Phi_{s, \cdot}^{\psi}(\nu)),$ then
\begin{equation}\label{e707}
d_t\left(\Phi_{s, \cdot}^\psi(\mu), \Phi_{s, \cdot}^\psi(\nu)\right) \leq(2 c)^{\frac{1}{2}}(t-s)^{\frac{1}{2}} d_t(\mu, \nu) .
\end{equation}
Now $\lambda=c,$ we take $t_0 \in (0,\frac{1}{2c})$ such that $(2c)t_0<1,$ then for any $s \in [0,T),$ the map $\Phi_{s,\cdot}^\psi$ in $M_{(s+t_0)\wedge T}$ is strictly compressed mapping under $d_t,$ so it has a unique fixed point.

\medskip
\textbf{Step 2}: Let $s=0, \psi:=\mathcal{L}_{X_0},$ according to Banach fixed point theorem, for any $t \in [0,t_0 \wedge T],$ exist unique $\mu(t)=\Phi_{0,t}^\psi(\mu).$ Meanwhile, from the definition of $\Phi_{s,t}^\psi,$ we can see that $X_{0,t}^{\psi,\mu}$ is the solution to (\ref{e702}) before $2t_0 \wedge T.$ If for every solution of (\ref{e702}) we take $\mu (t):= \mathcal{L}_{X(t)},$ we can infer that $\mu(t)$ is the solution of the following equation
\begin{equation}\label{e708}
\mu(t)=\Phi_{0, t}^\psi(\mu), t \in\left[0, t_0 \wedge T\right].
\end{equation}
Therefore, the uniqueness of the solution of equation (\ref{e702}) can be known from the uniqueness of the solution of equation (\ref{e708}).

It should be noted that if $t_0 \geq T,$ we complete the proof of the uniqueness and existence of equation (\ref{e702}). If $t_0 < T,$ because $t_0$ is independent of $X_0,$ we take $s=t_0, \psi = \mathcal{L}_{X(t_0)}.$ According to (\ref{e707}) we know that equation (\ref{e702}) exists a unique solution before $2t_0 \wedge T.$ Repeating the same steps a finite number of times, we can obtain the existence and uniqueness of the equation (\ref{e702}) solution before time $T.$
  \hspace{\fill}$\Box$

\subsection{Weak convergence analysis}\label{PT2.2}
\setcounter{equation}{0}
 \setcounter{definition}{0}


%

The proof of LDP is based on the weak convergence approach, and we first briefly provide the idea concerning this lengthy proof. Intuitively, as the parameter $\varepsilon \to 0$ in stochastic system (\ref{YE}), the noise term in equation (\ref{YE}) vanishes, then we obtain the following differential equation
\begin{equation}\label{X0}
\frac{d X^0(t)}{d t} = A (t, X^0(t),\mathcal{L}_{ X^0(t)}),\quad X^0(0) = x .
\end{equation}
We can know that (\ref{X0}) admits a unique solution by the hypothesis \ref{h1} (cf. \cite{Liu11}), always denoted by $X^0$ in this paper, which satisfies $X^0 \in C([0,T];H).$ In addition, we mention that the solution $X^0$ of (\ref{X0}) is a deterministic path and its distribution $\mathcal{L}_{X^0(t)} :=\delta_{X^0(t)},$ where $\delta_{X^0(t)}$ is the Dirac measure of $X^0(t).$

Recall (\ref{YE}) and for any $\mu \in C([0,T];\mathcal{P}_2(H))$ (which can be viewed as a deterministic measure flow), we consider the following reference SPDE
\begin{equation}\label{SPDE}
d\widetilde{X}^{\varepsilon}( t ) =A^{\mu ( t )}( t,\widetilde{X}^{\varepsilon}( t ) ) dt+\varepsilon \int_Z{f^{\mu ( t )}( t,\widetilde{X}^{\varepsilon}( t- ) ,z )}\widetilde{N}^{\varepsilon ^{-1}}( dt,dz ) ,
\end{equation}
where we denote $A^{\mu}( t, x ) = A( t, x ,\mu),$ the other term is similar. Note that (\ref{SPDE}) is a classical SPDE, one can apply the Yamada-Watanabe theorem so that there exists a measurable map $\mathcal{G}_\mu : S \to D([0,T];H) $ such that we have the representation
\begin{equation}\label{re}
\widetilde{X}^{\varepsilon}( t )=\mathcal{G}_\mu (\varepsilon N^{\varepsilon^{-1}}).
\end{equation}

Now we fix $\mu =\mu^\varepsilon:=\mathcal{L}_{X^\varepsilon},$ then (\ref{YE}) reduces to
\begin{equation}\label{fe00}
\left\{ \begin{aligned}	&dX^\varepsilon(t)=A^{\mu(t)}(t,X^\varepsilon(t))dt+\varepsilon \int_Z  f^{\mu(t)}(t,X^\varepsilon(t-),z)\widetilde{N}^{\varepsilon ^{-1}}(dt,dz),\\
	&X^\varepsilon(0)=x \in H.
\end{aligned} \right.
\end{equation}
We observe that $X^\varepsilon(t)$ is also a solution to system (\ref{SPDE}) with $\mu=\mu^\varepsilon,$ then by the strong uniqueness of system (\ref{SPDE}), in this case
$$
X^{\varepsilon}(t)=\widetilde{X}^{\varepsilon}(t), \quad t\in [0,T].
$$
Therefore, by the representation of (\ref{re}), we obtain
$$
X^\varepsilon =\mathcal{G}_\mu(\varepsilon N^{\varepsilon^{-1}})=\mathcal{G}_{\mu^\varepsilon}(\varepsilon N^{\varepsilon^{-1}}).
$$

For any $\varphi_\varepsilon,$ we define
$$
X^{\varphi_\varepsilon}:=\mathcal{G}^\varepsilon(\varepsilon N^{\varepsilon^{-1}\varphi_\varepsilon}),
$$
where $\mathcal{G}^\varepsilon$ is a measurable mapping from $\mathcal{M}_{FC}(Z_T)$ to $D([0,T];H).$ The process $X^{\varphi_\varepsilon}$ is the solution to following stochastic controlled equation
\begin{align}\label{CE}
X^{\varphi_\varepsilon}(t)=&x+\int_0^t A(s,X^{\varphi_\varepsilon}(s),\mathcal{L}_{X^{\varepsilon}(s)})ds
+\int_0^t\int_{Z}f(s,X^{\varphi_\varepsilon}(s),\mathcal{L}_{X^{\varepsilon}(s)},z)({\varphi_\varepsilon}(s,z)-1)\theta(dz)ds\nonumber\\
&+\varepsilon\int_0^t\int_{Z}f(s,X^{\varphi_\varepsilon}(s-),\mathcal{L}_{X^{\varepsilon}(s)},z)\widetilde{N} ^{\varepsilon ^{-1}\varphi_\varepsilon}(dz,ds).
\end{align}

We are now in the position to define the following skeleton equation
\begin{align}\label{SE}
X^g(t)=&x+\int_0^tA(s,X^g(s),\mathcal{L}_{X^0(s)})ds\nonumber\\
&+\int_0^t\int_{Z}f(s,X^g(s),\mathcal{L}_{X^0(s)},z)(g(s,z)-1)\theta(dz)ds,
\end{align}
where $g \in S.$ According to existence and uniqueness of the solution to skeleton equation (refer Lemma \ref{SEE} below), we know that there is a measurable mapping defined by
 $\mathcal{G}^0: S \to D([0,T];H)$
\begin{align}\label{G0}
\mathcal{G}^0\left( g \right) :=\left\{ \begin{array}{l}
	X^g,\ \ g\in S,\\
	0,\ \ \text{otherwise}.\\
\end{array} \right.
\end{align}

Set
$$
\mathcal{A}^N:=\Big\{\varphi \in \mathcal{A}\text{ and }\varphi(\omega) \in S^N,~~\mathbb{P}\text{-}a.s.\Big\}.
$$

Let $\{K_n \subset Z,n=1,2,...\}$ be an increasing sequence of compact sets of $Z$ such that $\bigcup_{n=1}^\infty K_n =Z.$ For each $n,$ let $K_n^c=Z\setminus K_n$ and

\begin{align*}
\mathcal{A}_{b,n}:=
&\Big\{ \varphi \in \mathcal{A} : \text{for all }(t,\omega) \in [0,T]\times \Omega,n\geq\varphi(t,x,\omega)\geq\frac{1}{n}\text{ if }x \in K_n \\
&\text{and }\varphi(t,x,\omega)=1\text{ if }x\in K_n^c\Big\},
\end{align*}	
and let $\mathcal{A}_b =\bigcup_{n=1}^\infty \mathcal{A}_{b,n}.$ Define $\widetilde{\mathcal{A}}^N=\mathcal{A}^N\cap\mathcal{A}_b.$

According to \cite[Theorem 4.1]{LSZZ}, Theorem \ref{LDP} is proved once we can clarify:

\noindent \textbf{Condition 4.1}:
\begin{enumerate}
\item [$(a)$]For any $N<\infty,$ let $g_n,$ $n\geq1,$ $g \in S^N$ be such that $g_n \to g$ as $n \to \infty.$ Then
$$
\mathcal{G}^0(g_n) \to \mathcal{G}^0(g)~~~~\text{in }D([0,T];H).
$$

\item [$(b)$]For any $N<\infty,$ let $\{\varphi_\varepsilon\}_{\varepsilon>0}\subset \widetilde{\mathcal{A}}^N$. Then, for any $c>0$
$$
\lim_{\varepsilon \to 0}\mathbb{P}\{\rho(X^{\varphi_\varepsilon},Y^{\varphi_\varepsilon})>c\}=0 ,
$$
where
$$
X^{\varphi_\varepsilon}:=\mathcal{G}^\varepsilon(\varepsilon N ^{\varepsilon ^{-1}\varphi_\varepsilon}),~~Y^{\varphi_\varepsilon}:=\mathcal{G}^0(\varphi_\varepsilon),
$$
and $\rho(\cdot,\cdot)$ stands for the Skorohod metric to the space $D([0,T];H).$
\end{enumerate}

In the following, we intend to verify the weak convergence criterion (i.e. Condition 4.1) for the above mentioned maps $\mathcal{G}^\varepsilon$ and $\mathcal{G}^0.$

\subsection{Skeleton and controlled equations}\label{Ske}

\begin{definition}
Giving $g \in S$, $X^g$ is called a solution to the skeleton equation (\ref{SE}) if the following conditions are satisfied
\begin{equation*}\label{e401}
X^g \in C([0,T];H) \cap L^\alpha([0,T];V),
\end{equation*} and for all $t \in [0,T]$
\begin{equation}\label{e402}
X^g(t)=x+\int_0^tA(s,X^g(s),\mathcal{L}_{X^0(s)})ds
+\int_0^t\int_{Z}f(s,X^g(s),\mathcal{L}_{X^0(s)},z)(g(s,z)-1)\theta(dz)ds,\mathbb{P}\text{-}a.s.
\end{equation}
holds in $V^*$.
\end{definition}

Before giving the prior estimates we need, the following lemma will be used in later proof.

\begin{lemma}\label{YEE}
Suppose that $(H1)$-$(H5)$ hold. Then, for each $x \in H,$ there is a unique solution to (\ref{YE01}) and exists a constant $C \in (0,\infty)$ such that
\begin{equation}\label{eYEE}
\mathbb{E}\Bigg[\sup_{s \in [0,t]}\|X(s)\|_H^2\Bigg] + 2\delta\mathbb{E}\Bigg[\int_0^t\|X(s)\|_V^\alpha ds \Bigg] \leq C e^{CT}(\|x\|_H^2 + 1).
\end{equation}
\end{lemma}
\begin{proof}
Applying It\^{o}'s formula, we have
\begin{align*}
\|X(t)\|_H^2= \|x\|_H^2+\int_0^t2_{V *}\langle A(s, X(s), \mathcal{L}_{X(s)}), X(s)\rangle_V d s
+I^1(t)+I^2(t), \quad t \in[0, T],
\end{align*}
where
\begin{align*}
I^1(t) & :=2\int_0^t \int_{Z}\langle X(s-), f(s, X(s-), \mathcal{L}_{X(s)}, z)\rangle_H \widetilde{N}(d s, d z) \\
I^2(t) & :=\int_0^t \int_{Z}[\|X(s-)+f(s, X(s-), \mathcal{L}_{X(s)}, z)\|_H^2-\|X(s-)\|_H^2 \\
& ~~~~-2\langle X(s-), f(s, X(s-), \mathcal{L}_{X(s)}, z)\rangle_H] N(d s, d z) .
\end{align*}
Define a stopping time
$$
\tau_R  :=\inf \left\{t \in[0, T]:\|X(t)\|_H>R\right\}, R>0.
$$

Then $\tau_R \uparrow T$ as $R \uparrow \infty$. By $(H5)(ii)$, we have

\begin{align}\label{l4.1e1}
\mathbb{E} I^1_{t \wedge \tau_R} & :=\mathbb{E}\Bigg[\int_0^{t \wedge \tau_R} \int_{Z}\langle X(s), f(s, X(s), \mathcal{L}_{X(s)}, z)\rangle_H \widetilde{N}(d s, d z) \Bigg]\nonumber\\
&=\mathbb{E}\Bigg[\int_0^{t \wedge \tau_R} \int_{Z}\langle X(s), f(s, X(s), \mathcal{L}_{X(s)}, z)\rangle_H \theta(d z) d s \Bigg]\nonumber\\
&<\infty .
\end{align}

From (\ref{l4.1e1}), we know that $I^1_{t \wedge \tau_R}$ is martingale. By $(H2)$, we obtain
\begin{align*}
& \mathbb{E}\Bigg[\sup _{s \in[0, t \wedge \tau_R]}\|X(s)\|_H^2\Bigg]+\delta \mathbb{E} \int_0^{t \wedge \tau_R}\|X(s)\|_V^\alpha d s \\
& \leq \|x\|_H^2+ C \mathbb{E} \int_0^{t \wedge \tau_R}\big(\|X(s)\|_H^2+\mathcal{L}_{X(s)}(\|\cdot\|_H^2)+1\big) d s
+ \mathbb{E} [I^2(t \wedge \tau_R)] .
\end{align*}
According to $(H5)(ii)$ and the Young's inequality, we have
\begin{align}\label{l4.1e2}
& \mathbb{E}[ I^2(t \wedge \tau) ] \nonumber\\
& = \mathbb{E} \Bigg[ \int_0^{t \wedge \tau} \int _z \big( \|X(s-)+ f(s,X(s-),\mathcal{L}_{X(s)},z)\|_H^2 - \|X(s-)\|_H^2 \nonumber\\
& \quad-2 \langle X(s-), f(s, X(s-), \mathcal{L}_{X(s)},z) \rangle _H \big) N(ds,dz) \Bigg] \nonumber\\
& \leq C \mathbb{E} \Bigg[ \int_0^{t \wedge \tau} \int_Z \|f(s,X(s-),\mathcal{L}_{X(s)},z)\|_H^2 N(ds,dz) \Bigg] \nonumber\\
& \leq C \mathbb{E} \Bigg[ \int_0^{t \wedge \tau} \int_Z l_2^2(s,z) \big(\|X(s)\|_H^2 + \mathcal{L}_{X(s)}(\|\cdot\|_H^2) +1 \big) \theta (dz) ds \Bigg] \nonumber\\
& \leq C \mathbb{E} \Bigg[ \big( \sup_{s \in [0, t \wedge \tau]} \|X(s)\|_H^2 +1 \big) \int_0^{t \wedge \tau} \int_Z l_2^2(s,z) \theta (dz) ds \Bigg] \nonumber\\
& \leq \frac{1}{2} \mathbb{E} \Bigg[ \sup_{s \in [0,t\wedge \tau]} \|X(s)\|_H^2 \Bigg] +C ,
\end{align}

where the constant $C > 0$ in each row may change. Then, we get
\begin{align*}
& \mathbb{E}\Bigg[\sup _{s \in[0, t \wedge \tau_R]}\|X(s)\|_H^2\Bigg]+2 \delta \mathbb{E} \int_0^{t \wedge \tau_R}\|X(s)\|_V^\alpha d s \\
& \leq C \|x\|_H^2 +C \mathbb{E} \int_0^t \mathbf{1}_{[0, \tau_R]}(s)(\|X(s)\|_H^2+\mathcal{L}_{X(s)}(\|\cdot\|_H^2)+1) d s \\
& \leq C \|x\|_H^2+C \int_0^t \mathbb{E}\|X(s)\|_H^2 d s . \\
\end{align*}

Taking $R \to \infty$, there is
\begin{align*}
& \mathbb{E}\Bigg[\sup _{s \in[0, t]}\|X(s)\|_H^2\Bigg]+2 \delta \mathbb{E} \int_0^t\|X(s)\|_V^\alpha d s \leq C \|x\|_H^2+C \int_0^t \mathbb{E}\Bigg[\sup _{r \in[0, s]}\|X(r)\|_H^2\Bigg] d s .
\end{align*}
Applying Gronwall's lemma gives that
\begin{align*}
& \mathbb{E}\left[\sup _{s \in[0, t]}\|X(s)\|_H^2\right]+2 \delta \mathbb{E} \int_0^t\|X(s)\|_V^\alpha d s \leq C e^{C T}\left(\|x\|_H^2+1\right) .
\end{align*}

At this point, we have completed the proof of Lemma \ref{YEE}.     \hspace{\fill}$\Box$
\end{proof}

\begin{lemma}\label{l4.2}
There is a constant $C_T>0$ such that
\begin{equation*}
\mathbb{E}\left[\sup _{t \in[0, T]}\left\|X^{\varepsilon}(t)-X^0(t) \right\|_H^2\right] \leq C_{T,\|x\|_H} .
\end{equation*}
\end{lemma}

\begin{proof}
For simplicity of notations, we denote $Z^{\varepsilon}(t):=X^{\varepsilon}(t)-X^0(t)$ which solves the following SPDE
$$
\left\{\begin{array}{l}
d Z^{\varepsilon}(t)=[A(t, X^{\varepsilon}(t), \mathcal{L}_{X^{\varepsilon}(s)})-A(t, X^0(t), \mathcal{L}_{X^{0}(t)})] d t\\
~~~~~~~~~~~+\varepsilon  \int_Z f(t, X^{\varepsilon}(t-), \mathcal{L}_{X^{\varepsilon}(t)}, z) \widetilde{N}^{\varepsilon^{-1}}(d t, d z), \\
Z^{\varepsilon}(0)=0 .
\end{array}\right.
$$
 Applying It\^{o}'s formula yields that
 \begin{align}\label{l4.2e1}
\|Z^{\varepsilon}(t)\|_H^2= & 2 \int_0^t \,_ {V^*}\langle A(s, X^{\varepsilon}(s), \mathcal{L}_{X^{\varepsilon}(s)})-A(s, X^0(s), \mathcal{L}_{X^{0}(s)}), Z^{\varepsilon}(s)\rangle_V d s \nonumber\\
& +2 \varepsilon \int_0^t \int_Z \langle f(s, X^{\varepsilon}(s-), \mathcal{L}_{X^{\varepsilon}(s)}, z), Z^{\varepsilon}(s) \rangle_H \widetilde{N}^{\varepsilon^{-1}}(d s, d z)\nonumber\\
& +\varepsilon^2 \int_0^t\int_Z  \|f(s, X^{\varepsilon}(s-), \mathcal{L}_{X^{\varepsilon}(s)}, z)\|_H^2 N^{\varepsilon^{-1}}(d s, d z)\nonumber\\
=: & \sum_{i=1}^3 I_i(t) .
\end{align}
Note that $\mathbb{W}_{2, H}\left(\mathcal{L}_{X^{\varepsilon}(s)}, \mathcal{L}_{X^{0}(s)}\right)^2 \leq \mathbb{E}\left\|X^{\varepsilon}(t)-X^0(t)\right\|_H^2$, in terms of the $(H3)$,  it follows that
\begin{align}\label{l4.2e2}
\mathbb{E} \Bigg[\sup_{t \in [0,T]} I_1(t) \Bigg]
&  \leq C \mathbb{E} \int_0^T\left(\|Z^{\varepsilon}(t)\|_H^2+\mathbb{W}_{2, H}(\mathcal{L}_{X^{\varepsilon}(s)}, \mathcal{L}_{X^{0}(s)})^2\right) d t\nonumber\\
& \leq C \int_0^T \mathbb{E}\|Z^{\varepsilon}(t)\|_H^2 d t  .
\end{align}
According to $(H5)(ii)$ and Lemma \ref{YEE}, the term $I_3(t)$ will be estimated as follows
\begin{align}\label{l4.2e3}
\mathbb{E} \Bigg[\sup_{t \in [0,T]} I_3(t) \Bigg]
& \textcolor{red} {\leq} \varepsilon \mathbb{E}\Bigg( \int_0^t\int_Z  \|f(s,X^{\varepsilon}(s), \mathcal{L}_{X^{\varepsilon}(s)}, z)\|_H^2 \theta (dz) ds\Bigg)\nonumber\\
& \leq C \varepsilon \mathbb{E}\Bigg( \int_0^t\int_Z  l^2_2(s,z) (\|X^{\varepsilon}(s)\|_H^2+\mathcal{L}_{X^{\varepsilon}(s)}(\|\cdot\|_H^2)+1) \theta (dz) ds\Bigg)\nonumber\\
& \leq C \varepsilon \mathbb{E}\Bigg[\bigg(\sup _{t \in[0, T]}\|X^\varepsilon(t)\|_H^2+1\bigg)\int_0^T \int_Z l^2_2(s,z) \theta (dz) ds \Bigg]\nonumber\\
& \leq C \varepsilon \mathbb{E}\Bigg(\sup _{t \in[0, T]}\|X^\varepsilon(t)\|_H^2+1\Bigg)\nonumber\\
& \leq \varepsilon C_{T,\|x\|_H}.
\end{align}
For the term $I_2(t)$, applying BDG's inequality and Young's inequality leads to
\begin{align}\label{l4.2e4}
& \mathbb{E} \Bigg[\sup_{t \in [0,T]} I_2(t) \Bigg]\nonumber\\
\leq & C \varepsilon \mathbb{E}\Bigg( \int_0^t\int_Z  \|f(s, X^{\varepsilon}(s-), \mathcal{L}_{X^{\varepsilon}(s)}, z)\|_H^2 \|Z^{\varepsilon}(s-)\|_H^2 N^{\varepsilon^{-1}}(d s, d z)\Bigg)^{\frac{1}{2}}\nonumber\\
\leq & \frac{1}{4} \mathbb{E} \Bigg[\sup_{t \in [0,T]} \|Z^{\varepsilon}(t)\|_H^2 \Bigg] +  C \varepsilon \mathbb{E}\Bigg( \int_0^t\int_Z  \|f(s,X^{\varepsilon}(s), \mathcal{L}_{X^{\varepsilon}(s)}, z)\|_H^2 \theta (dz) ds\Bigg) \nonumber\\
\leq & \frac{1}{4} \mathbb{E} \Bigg[\sup_{t \in [0,T]} \|Z^{\varepsilon}(t)\|_H^2 \Bigg]  + C \varepsilon \mathbb{E}\Bigg(\sup _{t \in[0, T]}\|X^\varepsilon(t)\|_H^2+1\Bigg)\nonumber\\
\leq & \frac{1}{4} \mathbb{E} \Bigg[\sup_{t \in [0,T]} \|Z^{\varepsilon}(t)\|_H^2 \Bigg] + \varepsilon C_{T,\|x\|_H}.
\end{align}
Thus combining (\ref{l4.2e2})-(\ref{l4.2e4}) with (\ref{l4.2e1}) and using Gronwall's inequality, we have
$$
\mathbb{E}\Bigg[\sup _{t \in[0, T]}\|Z_t^{\varepsilon}\|_H^2\Bigg] \leq C_{T,\|x\|_H},
$$
which completes the proof of Lemma \ref{l4.2}.
\hspace{\fill}$\Box$
\end{proof}

\begin{lemma}\label{lem4.2}
For any $\chi \in \mathcal{H}^{\varpi} \cap L_2(\theta_T)$, there exists a constant $C_{\chi,N}$ such that
\begin{equation}
C_{\chi,N}:= \sup_{g \in S^N} \int_0^T \int _{Z} \chi(s,z) |g(s,z)-1| \theta(dz) d s  < \infty.\label{e404}
\end{equation}
\end{lemma}
The proof of lemma \ref{lem4.2} can refer to \cite[Lemma 3.4]{BCD} or \cite[Lemma 4.1]{XZ}.

\begin{lemma}\label{SEE}
Suppose that $(H1)$-$(H5)$ hold. Then, for each $x \in H$ and $g\in S^N,$ there is a unique solution to (\ref{SE}) and exists a constant $C_{l_2,N,T} \in (0,\infty)$ such that
\begin{equation}\label{e403}
\sup_{t \in [0,T]}\|X^g(t)\|_H^2 + 2\delta\int_0^T\|X^g(t)\|_V^\alpha dt \leq C_{l_2,N,T}(\|x\|_H^2 + 1).
\end{equation}
\end{lemma}
\begin{proof}
Because the distribution of solutions to the skeleton equation (\ref{SE}) is fixed, we can treat it as a distribution independent equation, and the proof of the well-posedness part can refer in \cite[Theorem 1.2]{BLZ} or \cite[Proposition 5.1]{XZ}. Thus we only need to prove (\ref{e403}).

Recall the skeleton equation
\begin{equation*}
X^g(t)=x+\int_0^tA(s,X^g(s),\mathcal{L}_{X^0(s)})dt
+\int_0^t\int_{Z}f(s,X^g(s),\mathcal{L}_{X^0(s)},z)(g(s,z)-1)\theta(dz)ds,
\end{equation*}
Applying the chain rule to $\|X^g (t)\|_H^2$, we have
\begin{align*}
\|X^g(t)\|_H^2 = & \|x\|_H^2 + 2\int _0^t \, _{V^*} \langle A(s,X^g(s),\mathcal{L}_{X^0(s)}), X^g(s) \rangle _V d s   \nonumber\\
&+ 2 \int _0^t \int_Z \langle f(s,X^g(s),\mathcal{L}_{X^0(s)},z)(g(s,z)-1), X^g(s) \rangle _H \theta(d z) d s .
\end{align*}
According to $(H2)$, we have
\begin{align}\label{e414}
& \sup_{s \in [0,T]} \|X^g(s)\|_H^2 + \delta \int _0^T  \|X^g(s)\|_V^{\alpha} d s \nonumber\\
\leq & \|x\|_H^2 + C \int _0^T  \|X^g(s)\|_H^2 d s + CT \nonumber\\
& +2 \int_0^T \int _Z | \langle f(s,X^g(s),\mathcal{L}_{X^0(s)},z)(g(s,z)-1), X^g(s) \rangle_H |
\theta (dz) d s .
\end{align}
By Young's inequality, (\ref{e404}) and Cauchy-Schwarz's inequality, we get
\begin{align}\label{e415}
& 2 \int_0^T \int _Z | \langle f(s,X^g(s),\mathcal{L}_{X^0(s)},z)(g(s,z)-1), X^g(s) \rangle_H | \theta (dz) d s \nonumber\\
\leq & \sup_{s \in [0,T]} \|X^g(s)\|_H  \int_0^T \int _Z \|f(s,X^g(s),\mathcal{L}_{X^0(s)},z)\|_H |g(s,z)-1| \theta (dz) d s \nonumber\\
\leq & \frac{1}{2} \sup_{s \in [0,T]} \|X^g(s)\|_H^2 + 2\bigg(\int _0^T \int _Z \|f(s,X^g(s),\mathcal{L}_{X^0(s)},z)\|_H |g(s,z)-1| \theta (dz) d s\bigg)^2 \nonumber\\
\leq & \frac{1}{2} \sup_{s \in [0,T]} \|X^g(s)\|_H^2 \nonumber\\
 & + 2\bigg(\int _0^T \int _Z l_2(s,z)( \|X^g(s)\|_H  + \mathcal{L}_{X^0(s)}(\|\cdot\|_H^2)^ \frac{1}{2} + 1) |g(s,z)-1| \theta (dz) d s\bigg)^2 \nonumber\\
\leq & \frac{1}{2} \sup_{s \in [0,T]} \|X^g(s)\|_H^2 + 2 C_{l_2,N}\bigg(\int _0^T ( \|X^g(s)\|_H + \|X^0(s)\|_H + 1)   d s \bigg)^2 \nonumber\\
\leq & \frac{1}{2} \sup_{s \in [0,T]} \|X^g(s)\|_H^2 + 2 C_{l_2,N} \cdot T \int _0^T( \|X^g(s)\|_H + \|X^0(s)\|_H + 1)^2   d s \nonumber\\
\leq & \frac{1}{2} \sup_{s \in [0,T]} \|X^g(s)\|_H^2 +  C_{l_2,N,T} \int _0^T\|X^g(s)\|_H^2  d s + C_{l_2,N,T} + C_{l_2,N,T} \sup_{s \in [0,T]}\|X^0(s)\|_H^2.
\end{align}
By (\ref{e414}), (\ref{e415}) and Gronwall's lemma, we can obtain
$$
\sup_{t \in [0,T]}\|X^g(t)\|_H^2 + 2\delta\int_0^T\|X^g(t)\|_V^\alpha dt \leq C_{l_2,N,T}(\|x\|_H^2 + 1).
$$
This completes the proof of Lemma \ref{SEE}.    \hspace{\fill}$\Box$
\end{proof}

\begin{lemma}\label{L3}
There exists $\varepsilon_2 >0$ and a constant $C_{l_2,2,2, N, T,\|x\|_H}>0$ which is independent of $\varepsilon$ such that
\begin{equation}\label{e550}
\sup _{\varepsilon \in(0, \varepsilon_2]} \mathbb{E}\bigg(\sup _{t \in[0, T]}\|X^{\varphi_\varepsilon}(t)\|_H^2\bigg)+\mathbb{E}\bigg(\int_0^T\|X^{\varphi_\varepsilon}(t)\|_V^\alpha d t\bigg) \leq C_{l_2,2,2, N, T,\|x\|_H},
\end{equation}
where $X^{\varphi_\varepsilon}(t)$ satisfies the following stochastic controlled equation
\begin{align*}
X^{\varphi_\varepsilon}(t)=&x+\int_0^t A(s,X^{\varphi_\varepsilon}(s),\mathcal{L}_{X^{\varepsilon}(s)})dt
\nonumber\\
&+\int_0^t\int_{Z}f(s,X^{\varphi_\varepsilon}(s),\mathcal{L}_{X^{\varepsilon}(s)},z)({\varphi_\varepsilon}(s,z)-1)\theta(dz)ds\nonumber\\
&+\varepsilon\int_0^t\int_{Z}f(s,X^{\varphi_\varepsilon}(s-),\mathcal{L}_{X^{\varepsilon}(s)},z)\widetilde{N} ^{\varepsilon ^{-1}\varphi_\varepsilon}(dz,ds).
\end{align*}
\end{lemma}

\begin{proof}
By It\^{o}'s formula, we have
\begin{equation}\label{e551}
\left\|X^{\varphi_\varepsilon}(t)\right\|_H^2=\|x\|_H^2+\hat{I_1}(t)+\hat{I_2}(t)+\hat{I_3}(t)+\hat{I_4}(t),
\end{equation}
where
\begin{align*}
& \hat{I_1}(t)=2 \int_0^t \, _{V^*}\langle A(s,X^{\varphi_\varepsilon}(s),\mathcal{L}_{{X^\varepsilon}(s)}), X^{\varphi_\varepsilon}(s)\rangle_{V} d s ,\nonumber\\
& \hat{I_2}(t)=2 \int_0^t\int_Z\left\langle f(s, X^{\varphi_\varepsilon}(s),\mathcal{L}_{{X^\varepsilon}(s)}, z)\left(\varphi_{\varepsilon}(s, z)-1\right), X^{\varphi_\varepsilon}(s)\right\rangle_{H} \theta(d z) d s , \nonumber\\
& \hat{I_3}(t)=\int_0^t \int_Z\left[\|X^{\varphi_\varepsilon}(s-)+\varepsilon f(s, X^{\varphi_\varepsilon}(s-),\mathcal{L}_{{X^\varepsilon}(s)}, z)\|_H^2-\|X^{\varphi_\varepsilon}(s-)\|_H^2\right. \nonumber\\
&~~~~~~~~~~~~~~~~~~~~ \left.-2\left\langle\varepsilon f(s, X^{\varphi_\varepsilon}(s-),\mathcal{L}_{{X^\varepsilon}(s)}, z), X^{\varphi_\varepsilon}(s-)\right\rangle_{H}\right] N^{\varepsilon^{-1} \varphi_{\varepsilon}}(d z, d s), \nonumber\\
& \hat{I_4}(t)=2\varepsilon \int_0^t \int_Z \left\langle\varepsilon f(s, X^{\varphi_\varepsilon}(s-),\mathcal{L}_{{X^\varepsilon}(s)}, z), X^{\varphi_\varepsilon}(s-)\right\rangle_{H} \widetilde{N}^{\varepsilon^{-1}\varphi_{\varepsilon}} (d z, d s).
\end{align*}

By $(H2)$ and Lemma \ref{YEE} we can get
\begin{align}\label{e552}
\mathbb{E} \Bigg[\sup _{t \in [0,T]}\hat{I_1}(t) \Bigg] \leq &  \mathbb{E} \int_0^T \Big( C \big(\| X^{\varphi_\varepsilon}(s)\|_H^2 + \mathcal{L}_{X^{\varepsilon}(s)}(\|\cdot\|_H^2) +1\big) -\delta \|X^{\varphi_\varepsilon}(s)\|_V^\alpha \Big) d s \nonumber\\
\leq & -\delta \mathbb{E} \int_0^T\|X^{\varphi_\varepsilon}(s)\|_V^\alpha ds + C \mathbb{E} \int_0^T \| X^{\varphi_\varepsilon}(s)\|_H^2ds + C_{T,\|x\|_H} ,
\end{align}
and by $(H5)(ii)$
\begin{align}\label{e553}
& \mathbb{E} \Bigg[\sup _{t \in [0,T]} \hat{I_2}(t)\Bigg] \nonumber\\
\leq & 2 \mathbb{E} \int_0^T\int_{Z}\|f(s, X^{\varphi_\varepsilon}(s),\mathcal{L}_{{X^\varepsilon}(s)}, z)\|_H |\varphi_{\varepsilon}(s, z)-1| \|X^{\varphi_\varepsilon}(s)\|_H \theta(d z) d s \nonumber\\
\leq & 2 \mathbb{E} \int_0^T\int_Z l_2(s,z)(\| X^{\varphi_\varepsilon}(s)\|_H + \mathcal{L}_{X^{\varepsilon}(s)}(\|\cdot\|_H^2)^\frac{1}{2} + 1)
|\varphi_{\varepsilon}(s, z)-1| \|X^{\varphi_\varepsilon}(s)\|_H \theta(d z) d s \nonumber\\
\leq & 2 \mathbb{E} \int_0^T\|X^{\varphi_\varepsilon}(s)\|_H(\|X^{\varphi_\varepsilon}(s)\|_H+\|X^{\varepsilon}(s)\|_H+ 1 ) \int_Z l_2(s, z)|\varphi_{\varepsilon}(s, z)-1| \theta(d z) d s \nonumber\\
\leq & C_{T,\|x\|_H} \mathbb{E} \int_0^T (\|X^{\varphi_\varepsilon}(s)\|_H^2+1) \int_Z l_2(s, z)|\varphi_{\varepsilon}(s, z)-1| \theta(d z) d s \nonumber\\
\leq & C_{T,\|x\|_H} \mathbb{E} \int_0^T  \int_Z l_2(s, z)|\varphi_{\varepsilon}(s, z)-1| \theta(d z) d s \nonumber\\
& + C_{T,\|x\|_H} \mathbb{E} \int_0^T \|X^{\varphi_\varepsilon}(s)\|_H^2 \int_Z l_2(s, z)|\varphi_{\varepsilon}(s, z)-1| \theta(d z) d s   \nonumber\\
\leq & C_{l_2, N,T,\|x\|_H} +C_{l_2, N,T,\|x\|_H} \mathbb{E} \int_0^T \|X^{\varphi_\varepsilon}(s)\|_H^2 d s     .
 \end{align}

From \cite[Lemma 3.4]{BCD} or \cite[Lemma 4.1]{XZ}, we know that for any $h \in \mathcal{H}_2 \cap L_2(\theta_T),$ there exists a constant $C_{h,2,2,N}$ such that
\begin{equation}\label{e554}
C_{h, 2, 2, N}:=\sup _{g \in S^N} \int_0^T \int_Z h^2(s, z)(g(s, z)+1) \theta(d z) d s<\infty .
\end{equation}

Moreover, it is easy to see that (see e.g. \cite[(4.9)]{BLZ})
$$
\left|\|x+h\|_H^p-\|x\|_H^p-p\|x\|_H^{p-2}\langle x, h\rangle_{H}\right| \leq c_p\left(\|x\|_H^{p-2}\|h\|_H^2+\|h\|_H^p\right), \quad \forall x, h \in H .
$$
Let $p=2,$ then
$$
\hat{I_3}(t) \leq c_2 \int_0^t \int_Z\|\varepsilon f(s, X^{\varphi_\varepsilon}(s-),\mathcal{L}_{{X^\varepsilon}(s)}, z)\|_H^2 N^{\varepsilon^{-1} \varphi_{\varepsilon}}(d z, d s) .
$$
By $(H5)(ii)$ and (\ref{e554}), we can get
\begin{align}\label{e555}
& \mathbb{E} \Bigg[\sup_{t \in [0,T]}\hat{I}_3 (t)\Bigg]   \nonumber\\
\leq & \mathbb{E}\bigg(\int_0^T \int_Z C \|\varepsilon f(s, X^{\varphi_\varepsilon}(s-),\mathcal{L}_{{X^\varepsilon}(s)}, z)\|_H^2 N^{\varepsilon^{-1} \varphi_{\varepsilon}}(d z, d s)\bigg)  \nonumber\\
= & \varepsilon C \mathbb{E}\bigg(\int_0^T \int_Z\|f(s, X^{\varphi_\varepsilon}(s),\mathcal{L}_{{X^\varepsilon}(s)}, z)\|_H^2 \varphi_{\varepsilon}(s, z) \theta(d z) d s\bigg)  \nonumber\\
\leq & \varepsilon C \mathbb{E}\Bigg[\int_0^T \int_Z l_2^2(s, z) (\| X^{\varphi_\varepsilon}(s)\|_H + \mathcal{L}_{{X^\varepsilon}(s)}(\|\cdot\|_H^2)^\frac{1}{2} + 1 )^2 \varphi_{\varepsilon}(s, z) \theta(d z) d s\Bigg]  \nonumber\\
\leq & \varepsilon C \mathbb{E}\Bigg[\bigg(\sup _{s \in[0, T]}\|X^{\varphi_\varepsilon}(s)\|_H^2+ \sup _{s \in[0, T]}\|X^{\varepsilon}(s)\|_H + 1\bigg) \int_0^T \int_Z l_2^2(s, z) \varphi_{\varepsilon}(s, z) \theta(d z) d s\Bigg]  \nonumber\\
\leq & \varepsilon C_{l_2, 2, 2, N ,T,\|x\|_H} \mathbb{E}\bigg(\sup _{s \in[0, T]}\|X^{\varphi_\varepsilon}(s)\|_H^2\bigg)+\varepsilon C_{l_2, 2, 2, N, T,\|x\|_H} .
\end{align}

According to $(H5)(ii)$,  BDG's inequality (cf. \cite[Proposition 2.2]{WZ}) and (\ref{e554}), we have
\begin{align}\label{e556}
& \mathbb{E}\Bigg[\sup _{t \in[0, T]}\hat{I_4}(t)\Bigg] \nonumber\\
= & \mathbb{E}\Bigg[\sup_{t \in [0,T]} \int_0^t \int_Z 2 \langle \varepsilon f(s, X^{\varphi_\varepsilon}(s-),\mathcal{L}_{{X^\varepsilon}(s)}, z), X^{\varphi_\varepsilon}(s-) \rangle_H \widetilde{N}^{\varepsilon^{-1}\varphi_{\varepsilon}} (d z, d s) \Bigg] \nonumber\\
\leq & 4\varepsilon^2 \mathbb{E}\Bigg[ \int_0^T \int_Z \langle  f(s, X^{\varphi_\varepsilon}(s-),\mathcal{L}_{{X^\varepsilon}(s)}, z), X^{\varphi_\varepsilon}(s-) \rangle_H^2 N^{\varepsilon^{-1}\varphi_{\varepsilon}} (d z, d s) \Bigg] ^\frac{1}{2} \nonumber\\
\leq & 4\varepsilon^2 \mathbb{E}\Bigg[ \int_0^T \int_Z l_2^2(s, z) (\| X^{\varphi_\varepsilon}(s-)\|_H + \mathcal{L}_{{X^\varepsilon}(s)}(\|\cdot\|_H^2)^\frac{1}{2} + 1)^2 \| X^{\varphi_\varepsilon}(s-)\|_H^2 N^{\varepsilon^{-1}\varphi_{\varepsilon}} (d z, d s) \Bigg] ^\frac{1}{2} \nonumber\\
\leq & 4\varepsilon^2 \mathbb{E}\Bigg[ \sup_{s \in [0,T]} \| X^{\varphi_\varepsilon}(s)\|_H^2 \int_0^T \int_Z  l_2^2(s, z) ( \| X^{\varphi_\varepsilon}(s-)\|_H +\| X^{\varepsilon}(s-)\|_H  + 1)^2 N^{\varepsilon^{-1}\varphi_{\varepsilon}} (d z, d s) \Bigg] ^\frac{1}{2} \nonumber\\
\leq & \frac{1}{4} \mathbb{E}\Bigg[ \sup_{s \in [0,T]} \| X^{\varphi_\varepsilon}(s)\|_H^2 \Bigg] \nonumber\\
&+8\varepsilon C_{T,\|x\|_H} \mathbb{E}\Bigg[\bigg(\sup_{s \in [0,T]}\| X^{\varphi_\varepsilon}(s)\|_H^2 +1 \bigg) \int_0^T \int_Z  l_2^2(s, z) \varphi_{\varepsilon}(s, z) \theta(d z) d s \Bigg] \nonumber\\
\leq & \bigg(\frac{1}{4} + 4\varepsilon  C_{l_2,2,2,N,T,\|x\|_H} \bigg) \mathbb{E}\Bigg[ \sup_{s \in [0,T]} \| X^{\varphi_\varepsilon}(s)\|_H^2 \Bigg] +8\varepsilon C_{l_2,2,2,N,T,\|x\|_H} .
\end{align}

By using Gronwall's inequality and combining (\ref{e551})-(\ref{e556}), we have
$$
 \sup _{\varepsilon \in(0, \varepsilon_2]} \Bigg[\mathbb{E}\bigg(\sup _{t \in[0, T]}\|X^{\varphi_\varepsilon}(t)\|_H^2\bigg)+\mathbb{E}\bigg(\int_0^T\|X^{\varphi_\varepsilon}(t)\|_V^\alpha d t\bigg) \Bigg]
 \leq C_{l_2, N, T, \|x\|_H} < \infty.
$$
The proof is completed.    \hspace{\fill}$\Box$
\end{proof}

\subsection{Proof of Theorem \ref{LDP}}\label{ProTh}

In order to prove our main result, we recall the following
lemmas for late use, whose proof can refer \cite[Lemma 3.11]{BCD}.
\begin{lemma}\label{L1}
Let $\chi \in \mathcal{H}^{\infty} \cap L_2(\theta_T)$ and $N \in \mathbb{N},$ then for any $\varepsilon >0,$ there exists a compact set $K_\varepsilon \subset Z$ such that
\begin{equation}\label{e504}
\sup _{h \in S^N} \int_0^T \int_{K_{\varepsilon}^c} \chi(s, z)|h(s, z)-1| \theta(d z) d s \leq \varepsilon ,
\end{equation}
where $K_\varepsilon ^c$ is the complement of $K_\varepsilon.$
\end{lemma}

\begin{lemma}\label{L2}
Let $\chi \in \mathcal{H}^{\infty},$ for any compact set $K\subset Z$
\begin{equation}
 \lim _{J \rightarrow \infty} \sup _{h \in S^N} \int_0^T \int_K \chi(s, z)  1_{\left\{\chi(s, z) \geq J\right\}}(s, z) h(s, z) \theta(d z) d s=0 .
\end{equation}
\end{lemma}

\medskip
After obtaining the above preparations, we can now complete the proof of Theorem \ref{LDP} by verifying Condition 4.1, which will be divided into the following two steps.

\medskip

\textbf{Step 1}:
We now prove $(a)$ in Condition 4.1 first.

For any $N<\infty,$ let $g_n,$ $n\geq1,$ $g \in S^N$ be such that $g_n \to g$ as $n \to \infty.$ By Lemma \ref{SEE}, we know the following deterministic PDE
\begin{equation}\label{e506}
 \left\{\begin{array}{l}
d X^{g}(t)=A(t,X^{g}(t),\mathcal{L}_{X^0(t)})dt+\int_Z f\left(s, X^{g}(s),\mathcal{L}_{X^0(t)}, z\right)\left(g(s, z)-1\right) \theta(d z) d t, \\
X^{g}(0)=x \in H,
\end{array}\right.
\end{equation}
has a unique solution $X^g:=\mathcal{G}^0(g),$ and the following equation
\begin{equation}\label{e507}
 \left\{\begin{array}{l}
d X^{g_n}(t)=A(t,X^{g_n}(t),\mathcal{L}_{X^0(t)})dt+\int_Z f\left(s, X^{g_n}(s),\mathcal{L}_{X^0(t)}, z\right)\left(g_n(s, z)-1\right) \theta(d z) d t, \\
X^{g_n}(0)=x \in H,
\end{array}\right.
\end{equation}
has a unique solution $X^{g_n}:=\mathcal{G}^0(g_n).$

From (\ref{e506}) and (\ref{e507}), we can get
\begin{align}
& d\left(X^{g_n}(t)-X^g(t)\right) \nonumber\\
=
& \big[A(t,X^{g_n}(t),\mathcal{L}_{X^0(t)})-A(t,X^{g}(t),\mathcal{L}_{X^0(t)})\big] d t \nonumber\\
& +\int_Z \Big(f\left(s, X^{g_n}(s),\mathcal{L}_{X^0(t)}, z\right)\left(g_n(s, z)-1\right)-f\left(s, X^g(s),\mathcal{L}_{X^0(t)}, z\right)(g(s, z)-1) \Big)\theta(d z) d t .
\end{align}
Applying the chain rule to $\left\|X^{g_n}(t)-X^g(t)\right\|_H^2,$ we get
\begin{align}\label{e509}
& \|X^{g_n}(t)-X^g(t)\|_H^2 \nonumber\\
=& 2 \int_0^t \,_{V^*}\big\langle A(s,X^{g_n}(s),\mathcal{L}_{X^0(s)})-A(s,X^{g}(s),\mathcal{L}_{X^0(s)}), X^{g_n}(s)-X^g(s)\big\rangle_{V} d s \nonumber\\
& +2 \int_0^t\bigg\langle\int_Z f(s, X^{g_n}(s),\mathcal{L}_{X^0(s)}, z)(g_n(s, z)-1)-f(s, X^g(s), \mathcal{L}_{X^0(s)},z)(g(s, z)-1) \theta(d z), \nonumber\\
&~~~ X^{g_n}(s)-X^g(s)\bigg\rangle_{H} d s  \nonumber\\
\leq & C \int_0^t \|X^{g_n}(s)-X^g(s)\|_H^2 ds \nonumber\\
& +2 \int_0^t  \bigg\langle \int_Z f (s, X^{g}(s),\mathcal{L}_{X^0(s)}, z)(g_n(s, z)-g(s, z)) \theta(d z), X^{g_n}(s)-X^g(s) \bigg\rangle_H  ds \nonumber\\
& +2 \int_0^t\bigg\langle\int_Z (f(s, X^{g_n}(s),\mathcal{L}_{X^0(s)}, z)-f(s, X^g(s), \mathcal{L}_{X^0(s)},z))(g_n(s, z)-1) \theta(d z), \nonumber\\
&~~~ X^{g_n}(s)-X^g(s)\bigg\rangle_{H} d s  \nonumber\\
=:&C \int_0^t \|X^{g_n}(s)-X^g(s)\|_H^2 ds+Q_{n,1}(t)+Q_{n,2}(t) .
\end{align}

From $(H5)(i)$, for $Q_{n,2}(t),$ we can obtain
\begin{align}\label{e510}
|Q_{n,2}(t)|
&\leq 2 \int_0^t \int_Z \| X^{g_n}(s)-X^g(s) \| _H \nonumber\\
&~~~\cdot\|f(s, X^{g_n}(s),\mathcal{L}_{X^0(s)}, z)-f(s, X^g(s), \mathcal{L}_{X^0(s)},z)\|_H|g_n(s, z)-1|\theta(dz)d s \nonumber\\
&\leq 2 \int_Z l_1(s,z)|g_n(s, z)-1|\| X^{g_n}(s)-X^g(s) \| _H ^2 \theta(d z) d s .
\end{align}

Denote
\begin{equation}\label{e511}
h_n(s):=\int_Z l_1(s, z)\left|g_n(s, z)-1\right| \theta(d z).
\end{equation}
From (\ref{e404}), we know
\begin{equation}\label{e512}
 \sup _{n \geq 1} \int_0^T h_n(s) d s \leq C_{l_1, N} .
\end{equation}

Taking (\ref{e510})-(\ref{e511}) into (\ref{e509}), we obtain
\begin{align}
& \left\|X^{g_n}(t)-X^g(t)\right\|_{H}^2 \nonumber \\
\leq & c\int_0^t\left\|X^{g_n}(s)-X^g(s)\right\|_{H}^2 d s+Q_{n, 1}(t)+2 \int_0^t h_n(s) \cdot\left\|X^{g_n}(s)-X^g(s)\right\|_{H}^2 d s \text {. }
\end{align}

According to Gronwall's inequality and (\ref{e512}), we can obtain
\begin{align}\label{e514}
 \sup _{t \in[0, T]}\left\|X^{g_n}(t)-X^g(t)\right\|_{H}^2 &\leq \sup _{t \in[0, T]}\left|Q_{n, 1}(t)\right| \cdot e^{cT+2 \int_0^T h_n(s) d s} \nonumber\\
& \leq \sup _{t \in[0, T]}\left|Q_{n, 1}(t)\right| \cdot e^{cT+2 C_{l_1, N}}\nonumber \\
& =: C_{l_1, N, T} \cdot \sup _{t \in[0, T]}\left|Q_{n, 1}(t)\right| \text {. }
\end{align}

Now, let us estimate $Q_{n, 1}(t)$. By Lemma \ref{L1}, we know that for any $\varepsilon > 0,$ there exists a compact subset $K_\varepsilon$ of $Z$ such that (\ref{e504}) holds. We get
\begin{align}
 Q_{n, 1}(t)
&=2 \int_0^t \int_{K_{\varepsilon}}\left\langle f\left(s, X^g(s),  \mathcal{L}_{X^0(s)}, z\right)\left(g_n(s, z)-g(s, z)\right), X^{g_n}(s)-X^g(s)\right\rangle_{H} \theta(d z) d s \nonumber\\
&~~~ +2 \int_0^t \int_{K_{\varepsilon}^c}\left\langle f\left(s, X^g(s), \mathcal{L}_{X^0(s)}, z\right)\left(g_n(s, z)-g(s, z)\right), X^{g_n}(s)-X^g(s)\right\rangle_{H} \theta(d z) d s \nonumber\\
& =: I_{n, 1}(t)+I_{n, 2}(t) \text {. }
\end{align}

From Lemma \ref{SEE} and Lemma \ref{L1}, we know that, for any $t \in[0,T]$
\begin{align}
\left|I_{n, 2}(t)\right|
 &\leq 2 \int_0^T \int_{K_{\varepsilon}^c} l_2(s, z)(\left\|X^g(s)\right\|_{H}+\mathcal{L}_{X^0(s)}(\|\cdot\|_H^2)^{\frac{1}{2}}+1)  \nonumber\\
&~~~ \cdot \left|g_n(s, z)-g(s, z)\right| \cdot\left\|X^{g_n}(s)-X^g(s)\right\|_{H} \theta(d z) d s \nonumber\\
& \leq 2\sup _{s \in[0, T]} \big[ (\|X^g(s)\|_{H}+\|X^0(s)\|_{H}+1)(\|X^{g_n}(s)\|_{H}+\|X^g(s)\|_{H})\big]  \nonumber\\
&~~~ \cdot\Big(\int_0^T \int_{K_{\varepsilon}^c} l_2(s, z)\left|g_n(s, z)-1\right| \theta(d z) d s+\int_0^T \int_{K_{\varepsilon}^c} l_2(s, z)|g(s, z)-1| \theta(d z) d s\Big) \nonumber\\
& \leq C_{l_2,N,T} \varepsilon \left(\|x\|_H^2+1\right)   \text {. }
\end{align}

To estimate $I_{n,1}(t),$ we define
$$A_{2,J}=\Big\{(s,z) \in [0,T] \times Z:l_2(s,z) \geq J, J\in\mathbb{R}\Big\}.$$

In the following, for a subset $\mathcal{A} \subset [0,T] \times Z,$ we denote $\mathcal{A}^c$ the complement of $\mathcal{A}.$

Denote
\begin{align*}
 I_{n, 1, J}(t) := & 2 \int_0^t \int_{K_{\varepsilon}}\big\langle f(s, X^g(s), \mathcal{L}_{X^0(s)}, z)(g_n(s, z)-g(s, z)), \\
& X^{g_n}(s)-X^g(s)\big\rangle_{H} 1_{A_{2, J}}(s, z) \theta(d z) d s,\\
I_{n, 1, J^c}(t) := & 2 \int_0^t \int_{K_{\varepsilon}}\big\langle f(s, X^g(s),\mathcal{L}_{X^0(s)}, z)(g_n(s, z)-g(s, z)), \\
& X^{g_n}(s)-X^g(s)\big\rangle_{H} 1_{A_{2, J}^c}(s, z) \theta(d z) d s.
\end{align*}
It is obvious that
$$I_{n,t}(t)=I_{n, 1, J}(t)+I_{n, 1, J^c}(t).$$

Let us estimate $I_{n, 1, J}(t)$ and $I_{n, 1, J^c}(t)$ respectively. Notice that from (\ref{e403}), for any $t \in [0,T]$
\begin{align}\label{e517}
& \left|I_{n, 1, J}(t)\right| \nonumber\\
\leq & 2 \int_0^T \int_{K_{\varepsilon}} \|f(s,X^g(s),\mathcal{L}_{X^0(s)},z)\|_H |g_n(s, z)-g(s, z)|  \nonumber\\
&~~~ \cdot \left\|X^{g_n}(s)-X^g(s)\right\|_{H} 1_{A_{2, J}}(s, z) \theta(d z) d s \nonumber\\
\leq & 2 \int_0^T \int_{K_{\varepsilon}} l_2(s, z)\big(\left\|X^g(s)\right\|_{H}+\mathcal{L}_{X^0(s)}(\|\cdot\|_H^2)^{\frac{1}{2}}+1\big)\big(g_n(s, z)+g(s, z)\big)  \nonumber\\
&~~~ \cdot \big(\left\|X^{g_n}(s)\right\|_{H}+\left\|X^g(s)\right\|_{H}\big) 1_{A_{2, J}}(s, z) \theta(d z) d s \nonumber\\
\leq & 2 \sup _{s \in[0, T]}\Big[\big(\|X^g(s)\|_{H}+\|X^0(s)\|_{H}+1\big)\big(\left\|X^{g_n}(s)\right\|_{H}+\left\|X^g(s)\right\|_{H}\big)\Big] \nonumber\\
&~~~ \cdot \int_0^T \int_{K_{\varepsilon}} l_2(s, z)\big(g_n(s, z)+g(s, z)\big) 1_{A_{2, J}}(s, z) \theta(d z) d s \nonumber\\
\leq & C_{l_2,N,T}\left(\|x\|_H^2+1\right) \cdot \sup _{h \in S^N} \int_0^T \int_{K_{\varepsilon}} l_2(s, z) h(s, z) 1_{A_{2, J}}(s, z) \theta(d z) d s.
\end{align}
By Lemma \ref{L2}, we know that for $\varepsilon >0,$ there exists $J_\varepsilon > 0,$ such that
$$
\sup _{h \in S^N} \int_0^T \int_{K_{\varepsilon}} l_2(s, z) h(s, z) 1_{\left\{l_2(s, z) \geq J_{\varepsilon}\right\}} \theta(d z) d s \leq \frac{\varepsilon}{C_{l_2,N,T}\left(\|x\|_H^2+1\right)}.
$$
Choose $J$ in (\ref{e517}) to be $J_\varepsilon,$ then we have
\begin{equation}\label{e518}
\sup_{t \in [0,T]}|I_{n,1,J}(t)|\leq \varepsilon.
\end{equation}
According to (\ref{e514})-(\ref{e518}), we can obtain
\begin{align}\label{e519}
& \sup _{t \in[0, T]}\left\|X^{g_n}(t)-X^g(t)\right\|_{H}^2 \nonumber\\
\leq & C_{l_1, N, T} \cdot \sup _{t \in[0, T]}\left(\left|I_{n, 1, J_{\varepsilon}}(t)\right|+\left|I_{n, 1, J_{\varepsilon}^c}(t)\right|+\left|I_{n, 2}(t)\right|\right) \nonumber\\
\leq & C_{l_1, N, T} \cdot \sup _{t \in[0, T]}\left(\varepsilon+\left|I_{n, 1, J_{\varepsilon}^c}(t)\right|+ \varepsilon C_{l_2,N,T}(\|x\|_H^2+1) \right) .
\end{align}

To estimate $|I_{n, 1, J_{\varepsilon}^c}(t)|,$ we denote
$$U^n(s)=X^{g_n}(s)-X^g(s), \quad U^n(s(\delta))=X^{g_n}(s(\delta))-X^g(s(\delta)), \quad s \in [0,T],$$
where $s(\delta):= [\frac{s}{\delta}]\delta,$ $\delta > 0$ is small enough, and $[s]$ is the largest integer smaller than $s.$
Then
\begin{equation}\label{e520}
\sup _{t \in[0, T]}\left|I_{n, 1, J_{\varepsilon}^c}(t)\right| \leq 2 \sum_{i=1}^4 \tilde{I}_i,
\end{equation}
where
\begin{align*}
&\tilde{I}_1 :=\sup _{t \in[0, T]} \bigg| \int_0^t \int_{K_{\varepsilon}}\left\langle f(s, X^g(s),\mathcal{L}_{X^0(s)}, z)\big(g_n(s, z)-g(s, z)\big),\right. \\
&~~~~~~ \left.U^n\left(s\right)-U^n\left(s(\delta)\right)\right\rangle_{H} 1_{A_{2, J_{\varepsilon}}^c}(s, z) \theta(d z) d s \bigg|, \\
& \tilde{I}_2 :=\sup _{t \in[0, T]} \bigg| \int_0^t \int_{K_{\varepsilon}}\left\langle\left(f(s, X^g(s),\mathcal{L}_{X^0(s)}, z)-f(s, X^g(s(\delta)), \mathcal{L}_{X^0(s)},z)\right)\big(g_n(s, z)-g(s, z)\big),\right. \\
&~~~~~~~~ \left.U^n\left(s(\delta)\right)\right\rangle_{H} 1_{A_{2, J_{\varepsilon}}^c}(s, z) \theta(d z) d s \bigg|, \\
& \tilde{I}_3 :=\sup _{1 \leq k \leq 2^m} \sup _{t_{k-1} \leq t \leq t_k}\bigg| \int_{t_{k-1}}^t \int_{K_{\varepsilon}}\left\langle f(s, X^g(s(\delta)),\mathcal{L}_{X^0(s)}, z)\big(g_n(s, z)-g(s, z)\big),\right. \\
&~~~~~~~~ \left.U^n(s(\delta))\right\rangle_{H} 1_{A_{2, J_{\varepsilon}}^c}(s, z) \theta(d z) d s \bigg|, \\
& \tilde{I}_4 :=\sum_{k=1}^{2^m}\bigg|\int_{t_{k-1}}^{t_k} \int_{K_{\varepsilon}}\langle f(s, X^g(s(\delta)),\mathcal{L}_{X^0(s)}, z)\big(g_n(s, z)-g(s, z)\big), \nonumber\\
&~~~~~~~~ U^n(s(\delta))\rangle_{H} 1_{A_{2, J_{\varepsilon}}^c}(s, z) \theta(d z) d s\bigg| .
\end{align*}
Notice that $\tilde{I}_i,$ $i=1,...,4$ depend on $n,m,\varepsilon,$ etc., to shorten the notations, we omit these indices.

To make the proof process more clear, we give an estimate of $\tilde{I}_4$ in the following lemma.

\begin{lemma}\label{I4E}
For any constant $m$ large enough, we have
\begin{align}
\lim_{n \to \infty}\tilde{I}_4
=& \lim_{n \to \infty}\sum_{k=1}^{2^m}\bigg|\int_{t_{k-1}}^{t_k} \int_{K_{\varepsilon}}\langle f(s, X^g(s(\delta)),\mathcal{L}_{X^0(s)}, z)\big(g_n(s, z)-g(s, z)\big), \nonumber\\
&U^n(s(\delta))\rangle_{H} 1_{A_{2, J_{\varepsilon}}^c}(s, z) \theta(d z) d s\bigg| =0 .\label{re-lem4.5}
\end{align}
\end{lemma}
\begin{proof}
To estimate $\tilde{I}_4,$ for all $(s,z) \in [0,T]\times Z,$ we denote
$$
f_{\delta, n}(s, z)=\left\langle f\left(s, X^g\left(s(\delta)\right), \mathcal{L}_{X^0(s)}, z\right), U^n\left(s(\delta)\right)\right\rangle_{H} 1_{A_{2, J \varepsilon}^c}(s, z) .
$$
Since $U^n\left(s(\delta)\right)=X^{g_n}\left(s(\delta)\right)-X^g\left(s(\delta)\right)$ and (\ref{e403}), we know that
$$
\sup _{n \in \mathbb{N}} \sup _{s \in[0, T]}\left\|U^n\left(s(\delta)\right)\right\|_{H}  \leq 2 \sup _{h \in S^N} \sup _{s \in[0, T]}\|X^{h}(s)\|_H \leq \sqrt{C_{l_2,N, T}\left(\|x\|_H^2+1\right)}<\infty .
$$
 For any $s \in [0,T],$ $\{U^n(s(\delta))\}_{n\geq1}$ is weakly compact in $H$, hence there exists a subsequence (denoted as itself) and $U_k \in H,$ such that for all $\kappa \in H,$ we have
$$
\lim _{n \rightarrow \infty}\left\langle\kappa, U^n\left(s(\delta)\right)\right\rangle_{H}=\left\langle\kappa, U_k\right\rangle_{H}, \text { and }\left\|U_k\right\|_{H} \leq \sqrt{C_{l_2,N,T}\left(\|x\|_H^2+1\right)}<\infty .
$$
Therefore, on $[t_{k-1},t_k) \times K_\varepsilon$
\begin{equation}\label{e539}
\lim _{n \rightarrow \infty} f_{\delta, n}(s, z)=\left\langle f(s, X^g\left(s(\delta)\right), \mathcal{L}_{X^0(s)}, z), U_k\right\rangle_{H} 1_{A_{2, J_{\varepsilon}}^c}(s, z):=f_\delta(s, z), \theta(d z) d s\text{-}a . s . ,
\end{equation}
and
\begin{align}\label{e540}
 |f_{\delta, n}(s, z)|
 &\leq l_2(s, z)\big(\|X^g(s(\delta))\|_{H}+\mathcal{L}_{X^0(s)}(\|\cdot\|_H^2)^{\frac{1}{2}}+1\big)\|U^n(s(\delta))\|_{H} 1_{A_{2, J_{\varepsilon}}^c}(s, z) \nonumber\\
 &\leq J_{\varepsilon} C_{l_2,N,T}(\|x\|_H^2+1)<\infty,
\end{align}
which imply that
\begin{equation}\label{e541}
\left|f_\delta(s, z)\right| \leq J_{\varepsilon} C_{l_2,N, T}\left(\|x\|_H^2+1\right)<\infty .
\end{equation}

Similar to \cite[A.6]{BCD}, we can assume without loss of generality that
$$
m_{n, \varepsilon}:=\int_{t_{k-1}}^{t_k} \int_{K_{\varepsilon}} g_n(s, z) \theta(d z) d s \neq 0, \text { and } m_{\varepsilon}:=\int_{t_{k-1}}^{t_k} \int_{K_{\varepsilon}} g(s, z) \theta(d z) d s \neq 0 .
$$
Recall $ \theta_T=\lambda_T \otimes \theta $ and $\theta_T^g$  introduced in (\ref{e04}). Define probability measures $\tilde{\theta}_{n, \varepsilon},$ $\tilde{\theta}_{n}$ and $\nu_\varepsilon$ on $\left(\left[t_{k-1}, t_k\right) \times K_{\varepsilon}, \mathcal{B}\left(\left[t_{k-1}, t_k\right)\right) \otimes \mathcal{B}\left(K_{\varepsilon}\right)\right)$ as follows
\begin{align*}
& \tilde{\theta}_{n, \varepsilon}(\cdot)=\frac{1}{m_{n, \varepsilon}} \theta_T^{g_n}\big(\cdot \cap([t_{k-1}, t_k) \times K_{\varepsilon})\big), \\
& \tilde{\theta}_{\varepsilon}(\cdot)=\frac{1}{m_{\varepsilon}} \theta_T^g\big(\cdot \cap([t_{k-1}, t_k) \times K_{\varepsilon})\big), \\
& \nu_{\varepsilon}(\cdot)=\frac{\theta_T\big(\cdot \cap([t_{k-1}, t_k) \times K_{\varepsilon})\big)}{\theta_T\big([t_{k-1}, t_k) \times K_{\varepsilon}\big)} .
\end{align*}
Recall topology on $S^N$, see around (\ref{e04}). Since $g_n \to g$ as $n \to \infty$ in $S^N,$ we have that
\begin{equation}\label{e542}
\lim _{n \rightarrow \infty} m_{n, \varepsilon}=m_{\varepsilon},
\end{equation}
\begin{equation}\label{e543}
\tilde{\theta}_{n, \varepsilon}\text{ converges weakly to }\tilde{\theta}_{\varepsilon}\text{ as } n \to \infty.
\end{equation}
And there exists a constant $\alpha_\varepsilon$ such that the relative entropy function
\begin{align}\label{e544}
& \sup _{n \geq 1} R(\tilde{\theta}_{n, \varepsilon} \| \nu_{\varepsilon})\nonumber\\
= & \sup _{n \geq 1} \int_{[t_{k-1},t_k) \times {K_{\varepsilon}}}\bigg(\log \frac{d \tilde{\theta}_{n, \varepsilon}}{d \nu_\varepsilon}(x)\bigg) \tilde{\theta}_{n, \varepsilon}(dx)  \nonumber\\
= & \sup _{n \geq 1} \int_{t_{k-1}}^{t_k} \int_{K_{\varepsilon}} \log \bigg(\frac{\theta_T([t_{k-1}, t_k) \times K_{\varepsilon})}{m_{n, \varepsilon}} g_n(s, z)\bigg) \frac{1}{m_{n, \varepsilon}} g_n(s, z) \theta(d z) d s \nonumber\\
= & \sup _{n \geq 1}\bigg(\frac{1}{m_{n, \varepsilon}} \int_{t_{k-1}}^{t_k} \int_{K_{\varepsilon}}\log(g_n(s, z)) g_n(s, z) \theta(d z) d s \nonumber\\
& +\frac{1}{m_{n, \varepsilon}} \int_{t_{k-1}}^{t_k} \int_{K_{\varepsilon}}\log \frac{\theta_T([t_{k-1}, t_k) \times K_{\varepsilon})}{m_{n, \varepsilon}}g_n(s, z) \theta(d z) d s\bigg) \nonumber\\
= & \sup _{n \geq 1}\bigg(\frac{1}{m_{n, \varepsilon}} \int_{t_{k-1}}^{t_k} \int_{K_{\varepsilon}}(l(g_n(s, z))+g_n(s, z)-1) \theta(d z) d s \nonumber\\
& +\log \frac{\theta_T([t_{k-1}, t_k) \times K_{\varepsilon})}{m_{n, \varepsilon}}\bigg) \nonumber\\
\leq & \sup _{n \geq 1}\bigg(\frac{N}{m_{n, \varepsilon}}+1-\frac{\theta_T([t_{k-1}, t_k) \times K_{\varepsilon})}{m_{n, \varepsilon}}+\log \frac{\theta_T([t_{k-1}, t_k) \times K_{\varepsilon})}{m_{n, \varepsilon}}\bigg) \nonumber\\
\leq & \alpha_{\varepsilon}<\infty .
\end{align}

By (\ref{e539})-(\ref{e541}), (\ref{e543}), applying \cite[Lemma 2.8]{BD}, we have
\begin{enumerate}
\item [$(a)$]$\lim _{n \rightarrow \infty} \int_{t_{k-1}}^{t_k} \int_{K_{\varepsilon}} f_{\delta, n}(s, z) \tilde{\theta}_{n, \varepsilon}(d z d s)=\int_{t_{k-1}}^{t_k} \int_{K_{\varepsilon}} f_\delta(s, z) \tilde{\theta}_{\varepsilon}(d z d s)$,

\item [$(b)$] $\lim _{n \rightarrow \infty} \int_{t_{k-1}}^{t_k} \int_{K_{\varepsilon}} f_{\delta, n}(s, z) \tilde{\theta}_{\varepsilon}(d z d s)=\int_{t_{k-1}}^{t_k} \int_{K_{\varepsilon}} f_\delta(s, z) \tilde{\theta}_{\varepsilon}(d z d s)$,
\end{enumerate}
i.e.,

\begin{align*}
\left(a^{\prime}\right) &\lim _{n \rightarrow \infty} \int_{t_{k-1}}^{t_k} \int_{K_{\varepsilon}}\left\langle f\left(s, X^g\left(s(\delta)\right), \mathcal{L}_{X^0(s)}, z\right), U^n\left(s(\delta)\right)\right\rangle_{H} 1_{A_{2, J_{\varepsilon}}^c}(s, z) \frac{1}{m_{n, \varepsilon}} g_n(s, z) \theta(d z) d s \nonumber\\
= & \int_{t_{k-1}}^{t_k} \int_{K_{\varepsilon}}\left\langle f\left(s, X^g\left(s(\delta)\right), \mathcal{L}_{X^0(s)}, z\right), U_k\right\rangle_{H} 1_{A_{2, J_{\varepsilon}}^c}(s, z) \frac{1}{m_{\varepsilon}} g(s, z) \theta(d z) d s, \nonumber\\
\left(b^{\prime}\right) &\lim _{n \rightarrow \infty} \int_{t_{k-1}}^{t_k} \int_{K_{\varepsilon}}\left\langle f\left(s, X^g\left(s(\delta)\right), \mathcal{L}_{X^0(s)}, z\right), U^n\left(s(\delta)\right)\right\rangle_{H} 1_{A_{2, J_{\varepsilon}}^c}(s, z) \frac{1}{m_{\varepsilon}} g(s, z) \theta(d z) d s \nonumber\\
= & \int_{t_{k-1}}^{t_k} \int_{K_{\varepsilon}}\left\langle f\left(s, X^g\left(s(\delta)\right), \mathcal{L}_{X^0(s)}, z\right), U_k\right\rangle_{H} 1_{A_{2, J_{\varepsilon}}^c}(s, z) \frac{1}{m_{\varepsilon}} g(s, z) \theta(d z) d s .
\end{align*}

Therefore, taking (\ref{e542}) into account, we obtain
\begin{align}\label{e545}
& \lim _{n \rightarrow \infty}\bigg|\int_{t_{k-1}}^{t_k} \int_{K_{\varepsilon}}\langle f(s, X^g(s(\delta)),\mathcal{L}_{X^0(s)}, z), U^n(s(\delta))\rangle_{H}(g_n(s, z)-g(s, z)) 1_{A_{2, J_{\varepsilon}}^c}(s, z) \theta(d z) d s \bigg| \nonumber\\
= & \lim _{n \rightarrow \infty}  \bigg|m_{n,\varepsilon} \int_{t_{k-1}}^{t_k} \int_{K_{\varepsilon}}\langle f(s, X^g(s(\delta)),\mathcal{L}_{X^0(s)}, z), U^n(s(\delta))\rangle_{H} \frac{1}{m_{n,\varepsilon}} g_n(s, z) 1_{A_{2, J_{\varepsilon}}^c}(s, z) \theta(d z) d s  \nonumber\\
& ~~~~~~ - m_\varepsilon \int_{t_{k-1}}^{t_k} \int_{K_{\varepsilon}}\langle f(s, X^g(s(\delta)),\mathcal{L}_{X^0(s)}, z), U_k \rangle_{H} \frac{1}{m_{\varepsilon}} g(s, z) 1_{A_{2, J_{\varepsilon}}^c}(s, z) \theta(d z) d s   \bigg| \nonumber\\
= & \lim _{n \rightarrow \infty}  \bigg|m_{n,\varepsilon} \int_{t_{k-1}}^{t_k} \int_{K_{\varepsilon}}\langle f(s, X^g(s(\delta)),\mathcal{L}_{X^0(s)}, z), U^n(s(\delta))\rangle_{H} \frac{1}{m_{n,\varepsilon}} g_n(s, z) 1_{A_{2, J_{\varepsilon}}^c}(s, z) \theta(d z) d s  \nonumber\\
& ~~~~~~ - m_{n,\varepsilon} \int_{t_{k-1}}^{t_k} \int_{K_{\varepsilon}}\langle f(s, X^g(s(\delta)),\mathcal{L}_{X^0(s)}, z), U_k \rangle_{H} \frac{1}{m_{\varepsilon}} g(s, z) 1_{A_{2, J_{\varepsilon}}^c}(s, z) \theta(d z) d s   \bigg| \nonumber\\
&~~~ +  \lim _{n \rightarrow \infty} \bigg| m_{n,\varepsilon} \int_{t_{k-1}}^{t_k} \int_{K_{\varepsilon}}\langle f(s, X^g(s(\delta)),\mathcal{L}_{X^0(s)}, z), U_k \rangle_{H} \frac{1}{m_{\varepsilon}} g(s, z) 1_{A_{2, J_{\varepsilon}}^c}(s, z) \theta(d z) d s  \nonumber\\
& ~~~~~~ - m_{n,\varepsilon} \int_{t_{k-1}}^{t_k} \int_{K_{\varepsilon}}\langle f(s, X^g(s(\delta)),\mathcal{L}_{X^0(s)}, z), U^n (s(\delta))\rangle_{H} \frac{1}{m_{\varepsilon}} g(s, z) 1_{A_{2, J_{\varepsilon}}^c}(s, z) \theta(d z) d s \bigg| \nonumber\\
&~~~ +  \lim _{n \rightarrow \infty} \bigg| m_{n,\varepsilon} \int_{t_{k-1}}^{t_k} \int_{K_{\varepsilon}}\langle f(s, X^g(s(\delta)),\mathcal{L}_{X^0(s)}, z), U^n (s(\delta))\rangle_{H} \frac{1}{m_{\varepsilon}} g(s, z) 1_{A_{2, J_{\varepsilon}}^c}(s, z) \theta(d z) d s   \nonumber\\
& ~~~~~~ -  m_{\varepsilon} \int_{t_{k-1}}^{t_k} \int_{K_{\varepsilon}}\langle f(s, X^g(s(\delta)),\mathcal{L}_{X^0(s)}, z), U^n (s(\delta))\rangle_{H} \frac{1}{m_{\varepsilon}} g(s, z) 1_{A_{2, J_{\varepsilon}}^c}(s, z) \theta(d z) d s  \bigg| \nonumber\\
= & \lim _{n \rightarrow \infty} \bigg|\frac{m_{n,\varepsilon}-m_{\varepsilon}}{m_{\varepsilon}}\int_{t_{k-1}}^{t_k} \int_{K_{\varepsilon}}\langle f(s, X^g(s(\delta)),\mathcal{L}_{X^0(s)}, z), U^n(s(\delta))\rangle_{H}g(s, z) 1_{A_{2, J_{\varepsilon}}^c}(s, z) \theta(d z) d s \bigg|    \nonumber\\
= & 0 .
\end{align}
For any constant $m$ large enough, (\ref{e545}) implies \eref{re-lem4.5} holds.    \hspace{\fill}$\Box$
\end{proof}

Now, we continue the proof of $(a)$ in Condition 4.1. Let us estimate $\tilde{I}_i,$ $i=1,2,3,$ separately.
\begin{align}\label{e521}
& \tilde{I}_1 \leq \int_0^T \int_{K_{\varepsilon}} l_2(s, z) 1_{A_{2, J_{\varepsilon}}^c}(s, z)\big(\|X^g(s)\|_{H}+\mathcal{L}_{X^0(s)}(\|\cdot\|_H^2)^{\frac{1}{2}}+1\big)\|U^n(s)-U^n(s(\delta))\|_{H} \nonumber\\
&~~~~~~
\cdot\big(g_n(s, z)+g(s, z)\big) \theta(d z) d s \nonumber\\
& \leq \bigg( \sup _{s \in[0, T]}\|X^g(s)\|_{H}+ \sup _{s \in[0, T]}\|X^0(s)\|_{H}+ 1 \bigg) \nonumber\\
&~~~
\cdot J_{\varepsilon} \int_0^T \int_{K_{\varepsilon}}\|U^n(s)-U^n(s(\delta))\|_{H}\big(g_n(s, z)+g(s, z)\big) \theta(d z) d s.
\end{align}
Recall that from Remark 3.3 in \cite{BCD}, for any $a,b\in (0,\infty)$ and $\sigma \in [1,\infty)$
\begin{equation}\label{e522}
 a b \leq e^{\sigma a}+\frac{1}{\sigma}(b \log b-b+1)=e^{\sigma a}+\frac{1}{\sigma} l(b),
\end{equation}
where $l$ is defined as in (\ref{e01}), choose $a=1,b=g_n(s,z)$ or $g(s,z)$, by (\ref{e03}) and (\ref{e403}), (\ref{e521}) can be estimated by
\begin{align}\label{e523}
 \tilde{I}_1
&\leq \bigg(2 \sup _{s \in[0, T]}\|X^g(s)\|_{H}+ 1 \bigg) J_{\varepsilon} \cdot \int_0^T \int_{K_{\varepsilon}}\|U^n(s)-U^n(s(\delta))\|_{H} \nonumber \\
&~~~ \cdot \bigg(2 e^\sigma+\frac{1}{\sigma} l\big(g_n(s, z)\big)+\frac{1}{\sigma} l\big(g(s, z)\big)\bigg) \theta(d z) d s \nonumber\\
& \leq \bigg(2 \sup _{s \in[0, T]}\|X^g(s)\|_{H}+ 1\bigg) J_{\varepsilon} \cdot 2 e^\sigma \int_0^T \int_{K_{\varepsilon}}\|U^n(s)-U^n(s(\delta))\|_{H} \theta(d z) d s \nonumber\\
&~~~ +\frac{4}{\sigma} \bigg(2 \sup _{s \in[0, T]}\|X^g(s)\|_{H}+ 1 \bigg) J_{\varepsilon} \cdot \sup _{s \in[0, T]}\big(\|X^{g_n}(s)\|_{H}+\|X^g(s)\|_{H}\big) \nonumber\\
&~~~ \cdot \sup _{h \in S^N} \int_0^T \int_{K_{\varepsilon}} l\big(h(s, z)\big) \theta(d z) d s \nonumber\\
& \leq \bigg(2 \sup _{s \in[0, T]}\|X^g(s)\|_{H}+ 1 \bigg) J_{\varepsilon} \cdot 2 e^\sigma \theta(K_{\varepsilon}) \cdot \sqrt{T}  \nonumber\\
&~~~ \cdot \Bigg[\bigg(\int_0^T\|X^{g_n}(s)-X^{g_n}(s(\delta))\|_{H}^2 d s\bigg)^{\frac{1}{2}}+\bigg(\int_0^T\|X^g(s)-X^g(s(\delta))\|_{H}^2 d s\bigg)^{\frac{1}{2}}\Bigg] \nonumber\\
&~~~ +16 \bigg(\sup _{s \in[0, T]} \sup _{h \in S^N}\|X^{h}(s)\|_{H}^2 +1\bigg) J_{\varepsilon} \cdot \frac{N}{\sigma} \nonumber\\
& \leq C_{l_2,N,T}(\|x\|_H^2+1) J_{\varepsilon} \cdot e^\sigma \theta(K_{\varepsilon})  \nonumber \\
&~~~ \cdot \Bigg[\bigg(\int_0^T\|X^{g_n}(s)-X^{g_n}(s(\delta))\|_{H}^2 d s\bigg)^{\frac{1}{2}}+\bigg(\int_0^T\|X^g(s)-X^g(s(\delta))\|_{H}^2 d s\bigg)^{\frac{1}{2}}\Bigg] \nonumber\\
&~~~ +C_{l_2,N,T}(\|x\|_H^2+1) J_{\varepsilon} \cdot \frac{1}{\sigma} .
\end{align}
To estimate $\int_0^T\|X^{g_n}(s)-X^{g_n}(s(\delta))\|_H^2 ds,$ notice that from (\ref{e507}) we have
\begin{align}
X^{g_n}(s)-X^{g_n}(s(\delta))=&\int_{s(\delta)}^{s} A(t,X^{g_n}(t),\mathcal{L}_{X^0(t)})d t \nonumber\\
& +\int_{s(\delta)}^{s} \int_Z f(s, X^{g_n}(s),\mathcal{L}_{X^0(t)}, z)\big(g_n(s, z)-1\big) \theta(d z) d t .
\end{align}
Applying the chain rule to $\|X^{g_n}(s)-X^{g_n}(s(\delta))\|_H^2,$ we get
\begin{align}\label{e525}
& \|X^{g_n}(s)-X^{g_n}(s(\delta))\|_H^2\nonumber\\
= & 2 \int_{s(\delta)}^{s} \, _{V^*}\big\langle A(t,X^{g_n}(t),\mathcal{L}_{X^0(t)}), X^{g_n}(t)-X^{g_n}(s)\big\rangle_{V} d t \nonumber\\
&~~~ +2 \int_{s(\delta)}^{s}\left\langle\int_Z f(s, X^{g_n}(s),\mathcal{L}_{X^0(t)}, z)\big(g_n(s, z)-1\big) \theta(d z),X^{g_n}(t)-X^{g_n}(s)\right\rangle_{H} d t .
\end{align}
Integrating (\ref{e525}) over $[0,T]$ with respect to $s,$ we have
\begin{align}\label{e526}
& \int_0^T\|X^{g_n}(s)-X^{g_n}(s(\delta))\|_H^2 d s \nonumber\\
= & 2\int_0^T \int_{s(\delta)}^{s}\, _{V^*}\langle A(t,X^{g_n}(t),\mathcal{L}_{X^0(t)}), X^{g_n}(t)-X^{g_n}(s)\rangle_{V} d t d s \nonumber\\
& + 2\int_0^T \int_{s(\delta)}^{s}\left\langle\int_Z f(s, X^{g_n}(s),\mathcal{L}_{X^0(t)}, z)\big(g_n(s, z)-1\big) \theta(d z),X^{g_n}(t)-X^{g_n}(s)\right\rangle_{H} d t d s  \nonumber\\
=: &\tilde{I}_1^1+\tilde{I}_1^2 .
\end{align}
Then, according to the condition $(H4)$, we can obtain
\begin{align}\label{e5261}
 & \tilde{I}_1^1
 \leq  2\int_0^T \int_{s(\delta)}^{s} \|A(t,X^{g_n}(t)\mathcal{L}_{X^0(t)})\|_{V^*}
\|X^{g_n}(t)-X^{g_n}(s)\|_V d t d s \nonumber\\
\leq & 2 \Bigg [\int_0^T \int_{s(\delta)}^{s} \|A(t,X^{g_n}(t),\mathcal{L}_{X^0(t)})\|_{V^*}^{\frac{\alpha}{\alpha-1}}d t d s \Bigg ]^{\frac{\alpha-1}{\alpha}}
\Bigg[\int_0^T \int_{s(\delta)}^{s} \|X^{g_n}(t)-X^{g_n}(s)\|_V ^\alpha d t d s\Bigg] ^{\frac{1}{\alpha}}\nonumber\\
\leq & 2 \Bigg[\int_0^T \int_{s(\delta)}^{s} C(\|X^{g_n}(t)\|_V^\alpha+\mu(\|\cdot\|_H^2)+1)d t d s \Bigg]^{\frac{\alpha-1}{\alpha}}
\Bigg[\int_0^T \int_{s(\delta)}^{s} \|X^{g_n}(t)-X^{g_n}(s)\|_V ^\alpha d t d s\Bigg] ^{\frac{1}{\alpha}}\nonumber\\
\leq & C \Bigg[ \delta\bigg(\int_0^T \|X^{g_n}(t)\|_V^\alpha d t \bigg)  +\delta T \bigg(\sup_{s\in[0,T]}\|X^{g_n}(s)\|_H^2 + 1 \bigg)\Bigg]^{\frac{\alpha-1}{\alpha}} \cdot
\Bigg[\int_0^T  \|X^{g_n}(t)\|_V^\alpha d t \cdot \delta\Bigg]^{\frac{1}{\alpha}} \nonumber\\
\leq & C \Big[ C_{l_2,N,T}(\|x\|_H^2+1) (\delta+\delta T)\Big]^{\frac{\alpha-1}{\alpha}}\Bigg[ C_{l_2,N,T}(\|x\|_H^2+1) \delta\Bigg]^{\frac{1}{\alpha}}  .
\end{align}
In terms of the condition $(H5)(ii)$, it follows that
\begin{align}\label{e5262}
 &\tilde{I}_1^2
\leq  2\int_0^T \int_{s(\delta)}^{s}\int_Z \|f(s, X^{g_n}(s),\mathcal{L}_{X^0(t)}, z)\|_H|g_n(s, z)-1| \|X^{g_n}(t)-X^{g_n}(s)\|_H \theta(d z)d t d s  \nonumber\\
\leq & 2\int_0^T \int_{s(\delta)}^{s}\int_Z l_2(s,z)(\|X^{g_n}(t)\|_H+ \mathcal{L}_{X^0(t)}(\|\cdot\|_H^2)^{\frac{1}{2}}+ 1 )|g_n(s, z)-1|  \nonumber\\
& \cdot \|X^{g_n}(t)-X^{g_n}(s)\|_H \theta(d z)d t d s  \nonumber\\
\leq & 2\int_0^T \int_{s(\delta)}^{s}\big(\|X^{g_n}(t)\|_H + \mathcal{L}_{X^0(t)}(\|\cdot\|_H^2)^{\frac{1}{2}}+ 1 \big) \big(\|X^{g_n}(t)\|_H+\|X^{g_n}(s)\|_H\big) \nonumber\\
& \cdot \int_Z l_2(s,z)|g_n(s, z)-1| \theta(d z)d t d s  \nonumber\\
\leq & 4\bigg( \sup_{s \in [0,T]}\|X^{g_n}(s)\|_H+\sup_{s \in [0,T]}\|X^{0}(s)\|_H + 1\bigg) \nonumber\\
& \cdot \sup_{s \in [0,T]}\|X^{g_n}(s)\|_H \int_0^T \int_{s(\delta)}^{s}\int_Z l_2(s,z)|g_n(s, z)-1| \theta(d z)d t d s  \nonumber\\
\leq & C_{l_2,N,T}(\|x\|_H^2+1)\sup_{h \in S^N} \sup_{ s \in [0,T]} \int_{s(\delta)}^{s} \int_Z l_2(t,z)|h(t,z)|\theta (dz)dt .
\end{align}
Combining (\ref{e5261}) with (\ref{e5262}) leads to
\begin{align}\label{e5263}
& \int_0^T\|X^{g_n}(s)-X^{g_n}(s(\delta))\|_H^2 d s \nonumber\\
\leq & \textcolor{red}{C} \Big[ C_{l_2,N,T}(\|x\|_H^2+1) (\delta+\delta T)\Big]^{\frac{\alpha-1}{\alpha}}\Bigg[ C_{l_2,N,T}(\|x\|_H^2+1) \delta\Bigg]^{\frac{1}{\alpha}} \nonumber\\
&~~~ +C_{l_2,N,T}(\|x\|_H^2+1)\sup_{h \in S^N} \sup_{ s \in [0,T]} \int_{s(\delta)}^{s} \int_Z l_2(t,z)|h(t,z)|\theta (dz)dt.
\end{align}
From \cite[(3.5)]{BCD}, for any $\chi \in \mathcal{H}^{\infty} \cap L_2(\theta_T)$ we have
\begin{equation}\label{e527}
\lim_{\delta \to 0} \sup_{h \in S^N} \sup_{|l-s|\leq \delta}\int_l^s\int_Z \chi (t,z)|h(s,z)-1|\theta(dz)dt=0,
\end{equation}
so
\begin{equation}\label{e528}
\lim_{m \to \infty}\sup_{n \in \mathbb{N}}\int_0^T\left\|X^{g_n}(s)-X^{g_n}\left(s(\delta)\right)\right\|_{H}^2 d s=0.
\end{equation}
Similarly
\begin{equation}\label{e529}
\lim_{m \to \infty}\int_0^T\left\|X^g(s)-X^g\left(s(\delta)\right)\right\|_{H}^2 d s=0.
\end{equation}
Taking (\ref{e528}) and (\ref{e529}) into (\ref{e523}), we can obtain
\begin{equation}\label{e530}
\overline{\lim_{\delta \to 0}}\sup_{n \in \mathbb{N}}\tilde{I}_1\leq C_{l_2,N, T}\left(\|x\|_H^2+1\right) J_{\varepsilon} \cdot \frac{1}{\sigma},
\end{equation}
since $\sigma$ is arbitrary in $[1,\infty),$ we have
 \begin{equation}\label{e531}
\lim_{\delta \to 0}\sup_{n \in \mathbb{N}}\tilde{I}_1 = 0 .
\end{equation}

To estimate $\tilde{I}_2,$ denote
$$
A_{1,J}=\Big\{(s,z) \in [0,T] \times Z: l_1(s,z) \geq J, J \in \mathbb{R}\Big\}.
$$
Then, by condition $(H5)(i)$, we obtain
\begin{align}\label{e532}
\tilde{I}_2
\leq & \int_0^T \int_{K_{\varepsilon}} l_1(s, z)\left\|X^g(s)-X^g\left(s(\delta)\right)\right\|_{H}\left\|U^n\left(s(\delta)\right)\right\|_{H}\left|g_n(s, z)-g(s, z)\right| \theta(d z) d s \nonumber\\
= & \int_0^T \int_{K_{\varepsilon}} l_1(s, z)\left\|X^g(s)-X^g\left(s(\delta)\right)\right\|_{H}\left\|U^n\left(s(\delta)\right)\right\|_{H}\left|g_n(s, z)-g(s, z)\right| \nonumber\\
& \cdot 1_{A_{1, J}}(s, z) \theta(d z) d s \nonumber\\
& +\int_0^T \int_{K_{\varepsilon}} l_1(s, z)\left\|X^g(s)-X^g\left(s(\delta)\right)\right\|_{H}\left\|U^n\left(s(\delta)\right)\right\|_{H}\left|g_n(s, z)-g(s, z)\right| \nonumber\\
& \cdot 1_{A_{1, J}^c}(s, z) \theta(d z) d s \nonumber\\
=: &\tilde{I}_2^1+\tilde{I}_2^2  .
\end{align}
In terms of the condition (\ref{e403}), it follows that
\begin{align}\label{e5321}
\tilde{I}_2^1 \leq & \sup _{\hbar \in S^N} \sup _{s \in[0, T]} 4\left\|X^{\hbar}(s)\right\|_{H}^2 \cdot \int_0^T \int_{K_{\varepsilon}} l_1(s, z)\left(\left|g_n(s, z)\right|+|g(s, z)|\right) 1_{A_{1, J}}(s, z) \theta(d z) d s \nonumber\\
\leq & C_{l_2,N,T}\left(\|x\|_H^2+1\right) \sup _{h \in S^N} \int_0^T \int_{K_{\varepsilon}} l_1(s, z) h(s, z) 1_{A_{1, J}}(s, z) \theta(d z) d s .
\end{align}
According to (\ref{e522}), it is obvious that
\begin{align}\label{e5322}
\tilde{I}_1^2
\leq & J \sup _{\hbar \in S^N} \sup _{s \in[0, T]} 2\left\|X^{\hbar}(s)\right\|_{H}
 \int_0^T \int_{K_{\varepsilon}}\left\|X^g(s)-X^g\left(s(\delta)\right)\right\|_{H}\nonumber\\
 & \cdot \left(\left|g_n(s, z)\right|+|g(s, z)|\right) \theta(d z) d s \nonumber\\
\leq & J C_{l_2,N,T}\left(\|x\|_H^2+1\right) \int_0^T \int_{K_{\varepsilon}}\left\|X^g(s)-X^g\left(s(\delta)\right)\right\|_{H} \nonumber\\
&\cdot \Big(2e ^ \sigma + \frac{1}{\sigma} l\big(g_n(s, z)\big)+\frac{1}{\sigma} l\big(g(s, z)\big)\Big) \theta(dz)ds  \nonumber\\
\leq & J C_{l_2,N, T}\left(\|x\|_H^2+1\right) e^\sigma \theta\left(K_{\varepsilon}\right) \int_0^T\left\|X^g(s)-X^g\left(s(\delta)\right)\right\|_{H} d s \nonumber\\
& +J C_{l_2,N,T}\left(\|x\|_H^2+1\right) \int_0^T \int_{K_{\varepsilon}} \left\|X^g(s)-X^g\left(s(\delta)\right)\right\|_{H}  \nonumber\\
&~~~ \cdot \Big(\frac{1}{\sigma} l\big(g_n(s, z)\big)+\frac{1}{\sigma} l\big(g(s, z)\big)\Big) \theta(d z) d s .
\end{align}
Substituting (\ref{e5321}) and (\ref{e5322}) into (\ref{e532})
\begin{align}\label{e5323}
\tilde{I}_2
\leq & C_{l_2,N,T}\left(\|x\|_H^2+1\right) \sup _{h \in S^N} \int_0^T \int_{K_{\varepsilon}} l_1(s, z) h(s, z) 1_{A_{1, J}}(s, z) \theta(d z) d s \nonumber\\
& + J C_{l_2,N, T}\left(\|x\|_H^2+1\right) e^\sigma \theta\left(K_{\varepsilon}\right) \int_0^T\left\|X^g(s)-X^g\left(s(\delta)\right)\right\|_{H} d s \nonumber\\
& +J C_{l_2,N,T}\left(\|x\|_H^2+1\right) \int_0^T \int_{K_{\varepsilon}} \left\|X^g(s)-X^g\left(s(\delta)\right)\right\|_{H}  \nonumber\\
&~~~ \cdot \Big(\frac{1}{\sigma} l\big(g_n(s, z)\big)+\frac{1}{\sigma} l\big(g(s, z)\big)\Big) \theta(d z) d s
\nonumber\\
=: &\tilde{I}_{21}+\tilde{I}_{22}+\tilde{I}_{23}.
\end{align}
For the first term in the right hand-side of (\ref{e5323}), from Lemma \ref{L2}, we know that for the fixed $K_\varepsilon$ and any $\eta > 0,$ there exists $J_\eta > 0$ such that
\begin{equation}\label{e533}
\sup _{h \in S^N} \int_0^T \int_{K_{\varepsilon}} l_1(s, z) h(s, z) 1_{A_{1, J_\eta}}(s, z) \theta(d z) d s \leq \eta .
\end{equation}
Fixing the above $J_\eta.$ The second term in the right hand-side of (\ref{e5323}) can be controlled by
\begin{equation}\label{e534}
\tilde{I}_{22}\leq  J_\eta C_{l_2,N,T}(\|x\|_H^2+1) e^\sigma \theta(K_{\varepsilon})\bigg(\int_0^T\|X^g(s)-X^g(s(\delta))\|_{H}^2 d s \bigg)^{\frac{1}{2}} .
\end{equation}
From (\ref{e03}), we have
\begin{equation}\label{e535}
\tilde{I}_{23}\leq  C_{l_2,N,T}(\|x\|_H^2+1) J_\eta \cdot \frac{1}{\sigma} \cdot \bigg(\int_0^T\|X^g(s)-X^g(s(\delta))\|_{H}^2 d s\bigg)^{\frac{1}{2}}.
\end{equation}
Taking (\ref{e533})-(\ref{e535}) into (\ref{e5323}), and since $\sigma$ is arbitrary in $[1,\infty),$ by (\ref{e529})
\begin{equation*}
\overline{\lim _{\delta \to 0}} \sup _{n \in \mathbb{N}} \tilde{I}_2 \leq C_{l_2,N,T}\left(\|x\|_H^2+1\right) \eta .
\end{equation*}
Since $\eta$ is arbitrary in $(0 , \infty),$ we obtain
\begin{equation}\label{e536}
\lim _{\delta \to 0} \sup _{n \in \mathbb{N}} \tilde{I}_2=0 .
\end{equation}

By $(H5)(ii)$ and (\ref{e403})
\begin{align}\label{e537}
\tilde{I}_3 & \leq \sup _{1 \leq k \leq 2^m} \sup _{t_{k-1} \leq t \leq t_k} \int_{t_{k-1}}^t \int_{K_{\varepsilon}} l_2(s, z)\big(\|X^g(s(\delta))\|_{H}+ \mathcal{L}_{X^0(t)}(\|\cdot\|_H^2)^{\frac{1}{2}}+ 1 \big) \cdot\|U^n(s(\delta))\|_{H}   \nonumber\\
& ~~~ \cdot
\big(|g_n(s,z)-1|+|g(s,z)-1|\big)\theta(dz) d s  \nonumber\\
& \leq C_{l_2,N,T}(\|x\|_H^2+1) \sup _{h \in S^N} \sup _{1 \leq k \leq 2^m} \sup _{t_{k-1} \leq t \leq t_k} \int_{t_{k-1}}^t \int_{K_{\varepsilon}} l_2(s, z) \cdot|h(s, z)-1| \theta(d z) d s,
\end{align}
then applying (\ref{e527})
\begin{equation}\label{e538}
\lim _{m \rightarrow \infty} \sup _{n \in \mathbb{N}} \tilde{I}_3=0 .
\end{equation}

From (\ref{e531}), (\ref{e536}) and (\ref{e538}), we know that for any $\kappa > 0,$ there exists a $m_\kappa >0$ such that for all $m > m_\kappa$
\begin{equation}\label{e546}
 \sum_{i=1}^3 \sup _{n \in \mathbb{N}} \tilde{I}_i \leq \kappa .
\end{equation}
For the fixed $\kappa$ and $m_\kappa$ as above, we know from Lemma \ref{I4E} that
\begin{equation}\label{e547}
 \lim _{n \rightarrow \infty} \tilde{I}_4=0 .
\end{equation}
Taking (\ref{e546}) and (\ref{e547}) into account, from (\ref{e520}), we get
$$
\lim _{n \rightarrow \infty} \sup _{t \in[0, T]}\left|I_{n, 1, J_{\varepsilon}^c}(t)\right| \leq \kappa,
$$
since $\kappa$ is arbitrary in $(0,\infty)$, then
\begin{equation}\label{e548}
\lim _{n \rightarrow \infty} \sup _{t \in[0, T]}\left|I_{n, 1, J_{\varepsilon}^c}(t)\right| = 0.
\end{equation}
Taking (\ref{e548}) into account, from (\ref{e519}), we know that
$$
 \lim _{n \rightarrow \infty} \sup _{t \in[0, T]}\left\|X^{g_n}(t)-X^g(t)\right\|_{H}^2 \leq C_{l_1, N, T} \cdot\big(\varepsilon+ \varepsilon C_{l_2,N,T}(\|x\|_H^2+1) \big).
$$
Since $\varepsilon$ is arbitrary in $(0,\infty),$ it follows that
\begin{equation}\label{e549}
 X^{g_n} \rightarrow X^g \text { in } D\left([0, T] ; H\right),
\end{equation}
which indicates $(a)$ in Condition 4.1.

\medskip
\textbf{Step 2}: Next we will prove $(b)$ in Condition 4.1.

Let $N<\infty,$ $\{\varphi_\varepsilon \}_{\varepsilon >0} \subset \widetilde{\mathcal{A}}^N.$ From Lemma \ref{SEE}, we know that $Y^{\varphi_\varepsilon}:=\mathcal{G}^0(\varphi_\varepsilon)$ is the unique solution of the following equation
\begin{equation}\label{e557}
Y^{\varphi_{\varepsilon}}(t)=x+ \int_0^t A(s,Y^{\varphi_\varepsilon}(s),\mathcal{L}_{Y^0(s)}) d s+\int_0^t \int_Z f(s,Y^{\varphi_\varepsilon}(s),\mathcal{L}_{Y^0(s)},z)\big({\varphi_\varepsilon}(s,z)-1\big)\theta(dz) d s .
\end{equation}
From (\ref{CE}) and (\ref{e557}), we can get
\begin{align}\label{e558}
&X^{\varphi_\varepsilon}(t)-Y^{\varphi_{\varepsilon}}(t) \nonumber\\
= & \int_0^t \big(A(s,X^{\varphi_\varepsilon}(s),\mathcal{L}_{X^\varepsilon(s)})-A(s,Y^{\varphi_\varepsilon}(s),\mathcal{L}_{Y^0(s)})\big) d s \nonumber\\
& +\int_0^t \int_Z\big(f(s, X^{\varphi_\varepsilon}(s), \mathcal{L}_{X^\varepsilon(s)}, z)-f(s, Y^{\varphi_\varepsilon}(s), \mathcal{L}_{Y^0(s)}, z)\big)\big(\varphi_{\varepsilon}(s, z)-1\big) \theta(d z) d s \nonumber\\
&+\varepsilon \int_0^t \int_Z f(s, X^{\varphi_\varepsilon}(s-), \mathcal{L}_{X^\varepsilon(s)}, z) \widetilde{N}^{\varepsilon^{-1} \varphi_{\varepsilon}}(d z, d s) .
\end{align}

Applying It\^{o}'s formula to $\|X^{\varphi_\varepsilon}(t)-Y^{\varphi_{\varepsilon}}(t)\|_{H}^2 ,$ we have
\begin{align}\label{e559}
& \left\|X^{\varphi_\varepsilon}(t)-Y^{\varphi_{\varepsilon}}(t)\right\|_{H}^2 \nonumber \\
= & 2 \int_0^t \, _{V^*}\left\langle A(s,X^{\varphi_\varepsilon}(s),\mathcal{L}_{X^\varepsilon(s)})-A(s,Y^{\varphi_\varepsilon}(s),\mathcal{L}_{Y^0(s)}), X^{\varphi_\varepsilon}(s)-Y^{\varphi_{\varepsilon}}(s)\right\rangle_{V} d s \nonumber \\
& + 2 \int_0^t \int_Z\left\langle \big(f(s, X^{\varphi_\varepsilon}(s), \mathcal{L}_{X^\varepsilon(s)}, z)-f(s, Y^{\varphi_\varepsilon}(s), \mathcal{L}_{Y^0(s)}, z)\big)\big(\varphi_{\varepsilon}(s, z)-1\big),\right. \nonumber \\
&~~~~~ \left.X^{\varphi_\varepsilon}(s)-Y^{\varphi_{\varepsilon}}(s)\right\rangle_{H}  \theta(d z) d s     \nonumber \\
& + \varepsilon^2 \int_0^t \int_Z\left\|f\left(s, X^{\varphi_\varepsilon}(s-), \mathcal{L}_{X^\varepsilon(s)}, z\right)\right\|_{H}^2 N^{\varepsilon^{-1} \varphi_{\varepsilon}}(d z, d s) \nonumber \\
& +2 \varepsilon \int_0^t \int_Z\left\langle f(s, X^{\varphi_\varepsilon}(s-), \mathcal{L}_{X^\varepsilon(s)}, z), X^{\varphi_\varepsilon}(s-)-Y^{\varphi_{\varepsilon}}(s-)\right\rangle_{H} \widetilde{N}^{{\varepsilon^{-1}} \varphi_{\varepsilon}}(d z, d s) .
\end{align}

Then we will estimate the four terms in the right hand-side of (\ref{e559}) respectively.
For the first term, by Lemma \ref{l4.2} we can get
\begin{align}\label{e560}
& \mathbb{E} \Bigg[ \sup_{t \in [0,T] } 2 \int_0^t \, _{V^*}\left\langle A(s,X^{\varphi_\varepsilon}(s),\mathcal{L}_{X^\varepsilon(s)})-A(s,Y^{\varphi_\varepsilon}(s),\mathcal{L}_{Y^0(s)}), X^{\varphi_\varepsilon}(s)-Y^{\varphi_{\varepsilon}}(s)\right\rangle_{V} d s \Bigg]  \nonumber  \\
\leq & \mathbb{E} \int_0^T \Bigg[C \| X^{\varphi_\varepsilon}(s)-Y^{\varphi_{\varepsilon}}(s) ) \|_H^2  +C \mathbb{W}_{2,H}(\mathcal{L}_{X^\varepsilon(s)}, \mathcal{L}_{Y^0(s)})^2 \Bigg]ds \nonumber\\
\leq & C \mathbb{E} \int_0^T \Bigg[\|X^{\varphi_\varepsilon}(s)-Y^{\varphi_{\varepsilon}}(s) ) \|_H^2 + \Bigg] d s +C \mathbb{E} \int_0^T \Bigg[\|X^{\varepsilon}(s)-Y^{0}(s) ) \|_H^2 \Bigg] d s \nonumber\\
\leq & C \mathbb{E} \int_0^T \Bigg[\|X^{\varphi_\varepsilon}(s)-Y^{\varphi_{\varepsilon}}(s) ) \|_H^2 + \Bigg] d s + C_{T,\|x\|_H}.
\end{align}
For the second term, we obtian
\begin{align}\label{e561}
& \mathbb{E} \Bigg[  \sup_{t \in [0,T] } 2 \int_0^t \int_Z\left\langle \big(f(s, X^{\varphi_\varepsilon}(s), \mathcal{L}_{X^\varepsilon(s)}, z)-f(s, Y^{\varphi_\varepsilon}(s), \mathcal{L}_{Y^0(s)}, z)\big)\big(\varphi_{\varepsilon}(s, z)-1\big),\right. \nonumber \\
&~~~ \left.X^{\varphi_\varepsilon}(s)-Y^{\varphi_{\varepsilon}}(s)\right\rangle_{H}  \theta(d z) d s \Bigg]    \nonumber \\
& \leq 2 \mathbb{E} \int_0^T \int_Zl_1(s,z)\big(\| X^{\varphi_\varepsilon}(s) - Y^{\varphi_\varepsilon}(s) \|_H + \mathbb{W}_{2,H}(\mathcal{L}_{X^\varepsilon(s)}, \mathcal{L}_{Y^0(s)})\big) |\varphi_{\varepsilon}(s, z)-1|  \nonumber \\
&~~~ \cdot
\| X^{\varphi_\varepsilon}(s) - Y^{\varphi_\varepsilon}(s) \|_H \theta(d z) d s   \nonumber \\
& \leq 2 \mathbb{E} \int_0^T \big(\| X^{\varphi_\varepsilon}(s) - Y^{\varphi_\varepsilon}(s) \|_H + \mathbb{W}_{2,H}(\mathcal{L}_{X^\varepsilon(s)}, \mathcal{L}_{Y^0(s)})\big) \| X^{\varphi_\varepsilon}(s) - Y^{\varphi_\varepsilon}(s) \|_H  \nonumber \\
&~~~\cdot
\int_Z l_1(s,z) |\varphi_{\varepsilon}(s, z)-1| \theta(d z) d s   \nonumber \\
& \leq 4 \mathbb{E} \int _0^T   \| X^{\varphi_\varepsilon}(s) - Y^{\varphi_\varepsilon}(s) \|_H \mathbb{E} \|X^{\varepsilon}(s)-Y^{0}(s) ) \|_H   \int_Z l_1(t,z) |\varphi_{\varepsilon}(t, z)-1| \theta(d z) d t  \nonumber \\
& \leq 4 C_{T,\|x\|_H}\mathbb{E} \int _0^T   \| X^{\varphi_\varepsilon}(s) - Y^{\varphi_\varepsilon}(s) \|_H  h_{1, \varepsilon}(t)\theta(d z) d t,
\end{align}
where $h_{1, \varepsilon}(t) = \int_Z l_1(t,z) |\varphi_{\varepsilon}(t, z)-1| \theta(d z).$ According to (\ref{e404}), we have
\begin{equation}\label{e562}
\int_0^T h_{1, \varepsilon}(t) d t \leq C_{l_1, N} < \infty.
\end{equation}

For the third term, by $(H5)(ii)$ and (\ref{e554}) we get
\begin{align}\label{e563}
& \varepsilon^2  \mathbb{E}\Bigg[ \sup _{t \in[0, T]} \int_0^T \int_Z\left\|f\left(s, X^{\varphi_\varepsilon}(s-),\mathcal{L}_{X^\varepsilon(s)}, z\right)\right\|_{H}^2 N^{\varepsilon^{-1} \varphi_{\varepsilon}}(d z, d s)\Bigg] \nonumber\\
\leq & \varepsilon^2  \mathbb{E}\Bigg[ \int_0^T \int_Z\left\|f\left(s, X^{\varphi_\varepsilon}(s-),\mathcal{L}_{X^\varepsilon(s)}, z\right)\right\|_{H}^2 N^{\varepsilon^{-1} \varphi_{\varepsilon}}(d z, d s)\Bigg] \nonumber\\
= &  \varepsilon \mathbb{E}\Bigg[\int_0^T \int_Z\left\|f\left(s, X^{\varphi_\varepsilon}(s),\mathcal{L}_{X^\varepsilon(s)}, z\right)\right\|_{H}^2 \varphi_{\varepsilon}(s, z) \theta(d z) d s\Bigg] \nonumber\\
\leq &   \varepsilon \mathbb{E}\Bigg[\int_0^T \int_Z(\|X^{\varphi_\varepsilon}(s)\|_{H}^2+\mathcal{L}_{X^\varepsilon(s)}(\| \cdot \|_H^2)^{\frac{1}{2}}+ 1 )^2 l_2^2(s, z) \varphi_{\varepsilon}(s, z) \theta(d z) d s\Bigg] \nonumber\\
\leq & 2 \varepsilon C_{l_{2}, 2, 2, N,T,\|x\|_H} \mathbb{E}\bigg(\sup _{s \in[0, T]}\left\|X^{\varphi_\varepsilon}(s)\right\|_{H}^2+1\bigg) .
\end{align}
For the forth term, according to BDG's inequality, $(H5)(ii)$ and Young's inequality, we have
\begin{align}\label{e564}
& 2 \varepsilon \mathbb{E}\Bigg[\sup _{t \in[0, T]} \int_0^t \int_Z\langle f(s, X^{\varphi_\varepsilon}(s-),\mathcal{L}_{X^\varepsilon(s)}, z), X^{\varphi_\varepsilon}(s-)-Y^{\varphi_{\varepsilon}}(s-)\rangle_{H} \widetilde{N}^{\varepsilon^{-1} \varphi_{\varepsilon}}(d z, d s)\Bigg] \nonumber\\
\leq & 2 \varepsilon \mathbb{E}\Bigg[\int_0^T \int_Z\langle f(s, X^{\varphi_\varepsilon}(s-),\mathcal{L}_{X^\varepsilon(s)}, z), X^{\varphi_\varepsilon}(s-)-Y^{\varphi_{\varepsilon}}(s-)\rangle_{H}^2 N^{\varepsilon^{-1} \varphi_{\varepsilon}}(d z, d s)\Bigg]^{\frac{1}{2}} \nonumber\\
\leq & 2 \varepsilon \mathbb{E}\Bigg[\int_0^T \int_Z l_2^2(s, z)\big(\|X^{\varphi_\varepsilon}(s-)\|_{H}+\mathcal{L}_{X^\varepsilon(s)}(\|\cdot\|_H^2)^{\frac{1}{2}}+ 1 \big)^2 \nonumber\\
&~~~ \cdot \|X^{\varphi_\varepsilon}(s-)-Y^{\varphi_{\varepsilon}}(s-)\|_{H}^2 N^{\varepsilon^{-1} \varphi_{\varepsilon}}(d z, d s)\Bigg]^{\frac{1}{2}} \nonumber\\
\leq & 2 \varepsilon \mathbb{E}\Bigg[\int_0^T \int_Z 2 C_{T,\|x\|_H} l_2^2(s, z)\big(\|X^{\varphi_\varepsilon}(s-)\|_{H}^2+1\big) \cdot\|X^{\varphi_\varepsilon}(s-)-Y^{\varphi_{\varepsilon}}(s-)\|_{H}^2 N^{\varepsilon^{-1} \varphi_{\varepsilon}}(d z, d s)\Bigg]^{\frac{1}{2}} \nonumber\\
\leq & 4 C_{T,\|x\|_H} \varepsilon \mathbb{E}\Bigg[\sup _{s \in[0, T]}\|X^{\varphi_\varepsilon}(s)-Y^{\varphi_{\varepsilon}}(s)\|_{H}^2 \cdot \int_0^T \int_Z l_2^2(s, z)\big(\|X^{\varphi_\varepsilon}(s-)\|_{H}^2+1\big) N^{\varepsilon^{-1} \varphi_{\varepsilon}}(d z, d s)\Bigg]^{\frac{1}{2}} \nonumber\\
\leq & \frac{1}{2} \mathbb{E}\Bigg[\sup _{s \in[0, T]}\|X^{\varphi_\varepsilon}(s)-Y^{\varphi_{\varepsilon}}(s)\|_{H}^2\Bigg]
+8 \varepsilon C_{T,\|x\|_H} \mathbb{E}\Bigg[\int_0^T \int_Z l_2^2(s, z)\big(\|X^{\varphi_\varepsilon}(s)\|_{H}^2+1\big) \theta (dz)  d s \Bigg] \nonumber\\
\leq & \frac{1}{2} \mathbb{E}\Bigg[\sup _{s \in[0, T]}\|X^{\varphi_\varepsilon}(s)-Y^{\varphi_{\varepsilon}}(s)\|_{H}^2 \Bigg]\nonumber\\
& +8 \varepsilon C_{T,\|x\|_H} \mathbb{E}\Bigg[ \bigg( \sup _{s \in[0, T]} \|X^{\varphi_\varepsilon}(s)\|_{H}^2+1\bigg) \int_0^T \int_Z l_2^2(s, z) \theta (dz)  d s\Bigg] \nonumber\\
\leq & \frac{1}{2} \mathbb{E}\Bigg[\sup _{s \in[0, T]}\|X^{\varphi_\varepsilon}(s)-Y^{\varphi_{\varepsilon}}(s)\|_{H}^2\Bigg]
 +8 \varepsilon C_{l_2, 2, 2,N,T,\|x\|_H} \mathbb{E} \bigg( \sup _{s \in[0, T]} \|X^{\varphi_\varepsilon}(s)\|_{H}^2+1\bigg) .
\end{align}

From (\ref{e559})-(\ref{e564}), we have for any $\varepsilon \in (0,\frac{1}{2} e ^{(c + 4 C_{l_1,N})T} \wedge 1 ]$
\begin{align}\label{e565}
& \mathbb{E}\Bigg[\sup _{s \in[0, T]}\|X^{\varphi_\varepsilon}(s)-Y^{\varphi_{\varepsilon}}(s)\|_H^2\Bigg] \nonumber\\
\leq & \Big(2 \varepsilon C_{l_{2},2 ,2, N} +8 \varepsilon C_{{l_2},2,2,N}\Big)\mathbb{E} \Bigg[\sup _{s \in[0, T]}\|X^{\varphi_\varepsilon}(s)\|_{H}^2+1\Bigg] \cdot e^{cT+4 C_{l_1, N}} \nonumber\\
=: & \varepsilon C_{l_1,l_2,2,2,N,T} \mathbb{E}\bigg(\sup _{s \in[0, T]}\|X^{\varphi_\varepsilon}(s)\|_{H}^2+1 \bigg) .
\end{align}
Here the constant $ C_{l_1,l_2,2,2,N,T}$ is independent on $\varepsilon.$ According to Lemma \ref{L3} ,the inequality above implies that there exists $\varepsilon_0 >0$ small enough such that
\begin{equation}\label{e566}
\sup_{\varepsilon \in (0,\varepsilon_0)} \mathbb{E}\Bigg[\sup _{s \in[0, T]}\|X^{\varphi_\varepsilon}(s)\|_{H}^2\Bigg]\leq C_{l_2, N, T,\|x\|_H}< \infty .
\end{equation}

Combining (\ref{e565}) and (\ref{e566}) we know that
$$
X^{\varphi_\varepsilon} \longrightarrow Y^{\varphi_\varepsilon} \text { in } L^2(\Omega ; L^{\infty}([0, T] ; H)) \text { as } \varepsilon \rightarrow 0 ,
$$
which implies (b) in Condition 4.1.    The proof is complete. \hspace{\fill}$\Box$

\vspace{3mm}


\noindent\textbf{Data availability} Data sharing is not applicable to this article as no datasets were generated or analysed during
the current study.

\noindent\textbf{Statements and Declarations} On behalf of all authors, the corresponding author states that there is no conflict of interest.

\vspace{0.5cm}

\noindent\textbf{Acknowledgements} {The authors would like to thank professor Wei Liu for some help suggestions and valuable comments. This work is supported by NSFC (No.~12371147)  and the PAPD of Jiangsu Higher Education Institutions.}

\end{document}